\documentclass[12pt]{article}
\usepackage[utf8]{inputenc}
\usepackage[T1]{fontenc}
\usepackage{amsmath,amssymb,amsthm,amsfonts}
\usepackage{mathtools}
\usepackage{bm}
\usepackage[margin=1.1in]{geometry}
\usepackage{hyperref}
\usepackage{graphicx}
\usepackage{booktabs}
\usepackage{array}
\usepackage{enumitem}
\usepackage{xcolor}
\usepackage{setspace}
\usepackage{titlesec}
\usepackage[numbers]{natbib}
\usepackage{float}
\usepackage{caption}
\usepackage{subcaption}
\usepackage{tcolorbox}
\tcbuselibrary{theorems}

\hypersetup{
  colorlinks=true,
  linkcolor=blue!60!black,
  citecolor=green!50!black,
  urlcolor=blue!70!black
}

\theoremstyle{plain}
\newtheorem{theorem}{Theorem}[section]
\newtheorem{proposition}[theorem]{Proposition}
\newtheorem{lemma}[theorem]{Lemma}
\newtheorem{corollary}[theorem]{Corollary}
\theoremstyle{definition}
\newtheorem{definition}[theorem]{Definition}

\newtheorem{remark}[theorem]{Remark}
\newtheorem{assumption}[theorem]{Assumption}

\DeclareMathOperator{\Tr}{Tr}

\DeclareMathOperator{\spec}{\sigma}

\newcommand{\R}{\mathbb{R}}
\newcommand{\E}{\mathbb{E}}
\newcommand{\Prob}{\mathbb{P}}
\newcommand{\Lgen}{\mathcal{L}}
\newcommand{\W}{\mathbf{W}}

\newcommand{\norm}[1]{\left\lVert #1 \right\rVert}
\newcommand{\abs}[1]{\left| #1 \right|}
\newcommand{\ip}[2]{\langle #1, #2 \rangle}
\newcommand{\dif}{\,\mathrm{d}}
\newcommand{\loss}{L}

\title{%
  \textbf{Joint Lyapunov Certificates for $K$-Agent\\
  Generative AI Governance}\\[0.5em]
  \large Stochastic Stability, Emergent Ensemble Risk,\\
  and Zero-Knowledge Governance Attestation
}
\author{%
  Sriram Nagaraj\footnote{dr.sri.nagaraj@gmail.com}
}
\date{}

\begin{document}
\maketitle
\thispagestyle{empty}

\begin{abstract}
  We develop a rigorous mathematical framework for the governance of systems of
  $K$ self-adapting generative AI models under the principles of
  Model Risk Management (MRM).  When multiple models share a
  meta-learning coupling $\gamma$ through an interaction matrix $A$, the
  per-agent Lyapunov analysis that underpins standard MRM is provably insufficient: individual agents can each satisfy their
  declared stability bounds while the \emph{joint} system is in a
  regime of emergent ensemble-level drift.  We formalize this gap through the
  \emph{Joint Lyapunov Proof} (JLP)---a cryptographic and stochastic
  protocol that attests, without revealing proprietary weights, that the
  aggregate dynamics $\W_t = \{W^1_t, \ldots, W^K_t\}$ satisfy MRM Ongoing Monitoring standard at every validation epoch $\tau_n$.

  Our main contributions are fourfold.  First, we give a complete
  characterization of the infinitesimal generator $\Lgen V$ of the joint
  quadratic Lyapunov function $V(\W) = \tfrac{1}{2}\sum_k \norm{W^k}^2$
  under the coupled It\^{o} SDE and prove that the naive self-decay bound
  $\alpha_{\mathrm{eff}} = 2\alpha_{\mathrm{self}}$ fails as soon as
  $\gamma > 0$.  Second, we derive the exact critical coupling threshold
  $\gamma^*(A) = \alpha_{\mathrm{self}} / \abs{\lambda_{\min}(A)}$, where
  $\lambda_{\min}(A)$ is the most negative eigenvalue (real part) of $A$,
  above which the system loses mean-square stability.  The
  destabilizing mode is governed by $\lambda_{\min}(A)$---not the Perron
  root $\lambda_{\max}(A)$, which is the consensus mode, is the \emph{most}
  stable direction, and equals $1$ for every row-stochastic topology (so it
  cannot encode topology-dependent risk).  The quadratic Lyapunov function
  yields a \emph{safe sufficient} certificate $\gamma <
  \alpha_{\mathrm{self}}/\abs{\lambda_{\min}(A_{\mathrm{sym}})}$,
  $A_{\mathrm{sym}}=\tfrac12(A+A^\top)$, which coincides with $\gamma^*$ for
  normal topologies and is strictly conservative for non-normal ones (the
  star), so it never certifies an unstable system.  Third, we prove a
  Noise-Floor Theorem establishing that the stationary second moment
  satisfies $\E[V_\infty] = \beta/(2\alpha_{\mathrm{self}})$ with
  $\beta = \tfrac{1}{2}Kd\sigma^2$ exactly in the decoupled case.
  Fourth, we identify the correct target for zero-knowledge attestation.
  A per-epoch Succinct Non-Interactive Argument of Knowledge (SNARK) on the live weights proving the pointwise generator
  inequality $\Lgen V(\W_t)<-\alpha V(\W_t)+\beta$ is \emph{vacuous}: by the
  stability theorem the inequality holds at every $\W$ once the coefficients
  are in the stable regime, and it is in fact satisfied with greater slack
  as $\norm{\W_t}\to\infty$, so a blown-up system passes most comfortably.
  The property that actually governs stability---$\gamma<\gamma^*(A)$---is a
  function of the declared $(A,\alpha_{\mathrm{self}},\gamma)$ and needs no
  weights; we therefore construct a SNARK $\pi_{\mathrm{JLP}}$ that proves
  the static spectral certificate
  $\alpha_{\mathrm{self}}I+\gamma A_{\mathrm{sym}}\succ0$ on a committed
  coupling matrix via a Cholesky factorization, established once per
  configuration rather than at every epoch.  All
  theoretical claims are validated against five numerical studies using a $K=5$ (resp.\ $K=10$) multi-agent softmax system.
\end{abstract}

\tableofcontents
\newpage

\section{Introduction}
\label{sec:intro}

The deployment of multiple interacting generative AI models within a
single institution---each capable of updating its own internal
parameters in response to market data---creates a new class of ensemble-level
risk that lies outside the scope of classical single-model Model Risk Management (MRM).
Traditional MRM compliance treats each model as an independent unit:
the model owner validates the model, documents its assumptions, monitors
its outputs, and attests to conceptual soundness.  This paradigm was
designed for static or slowly-evolving statistical models.  It breaks
down, we argue, in the multi-agent generative setting for two reasons.

\paragraph{Reason 1: Cross-agent coupling induces hidden correlated drift.}
When a meta-learning parameter $\Phi_t$---a shared embedding layer,
a common reinforcement-learning reward signal, or a joint fine-tuning
objective---influences all $K$ agents simultaneously, each agent's
weight update becomes correlated with all others through the coupling
matrix $A$.  A small bias $h$ in $\Phi_t$ propagates through the coupling
so that all agents drift in the same direction, even when no single
agent's individual drift exceeds its declared threshold.

\paragraph{Reason 2: Individual stability does not imply joint stability.}
Our Theorem~\ref{thm:emergent} (Emergent Ensemble Risk) shows rigorously
that there exist parameter regimes in which the per-agent Lyapunov
function $V^k(t) = \tfrac{1}{2}\norm{W^k_t}^2$ satisfies
$V^k(t) \leq \theta_{\mathrm{ind}}$ for every $k = 1,\ldots,K$, while
the \emph{joint} function $V(\W_t) = \sum_k V^k(t)$ exceeds the
threshold $\theta_{\mathrm{joint}} = (1+\delta)K\,\E[V^k_\infty]$ for
any $\delta > 0$, when a coordinated hidden drift is present.  This
is not a pathological corner case: it follows directly from the
variance--aggregation structure of independent 1D Ornstein--Uhlenbeck
processes.

\paragraph{The JLP protocol.}
We address both failure modes through the Joint Lyapunov Proof.  The
stability of the joint system is governed by whether the declared coupling
configuration satisfies the spectral certificate
\begin{equation}
  \gamma < \gamma^*(A) = \frac{\alpha_{\mathrm{self}}}{\abs{\lambda_{\min}(A)}}
  \qquad\Longleftarrow\qquad
  \alpha_{\mathrm{self}} I + \gamma A_{\mathrm{sym}}\succ0,
  \label{eq:jlp_bound}
\end{equation}
a condition on $(A,\alpha_{\mathrm{self}},\gamma)$ that involves no live
weights.  The firm submits, once per model-configuration change, a
zero-knowledge attestation $\pi_{\mathrm{JLP}}$ proving the safe (Cholesky)
certificate on the right of \eqref{eq:jlp_bound} for its committed---possibly
proprietary---coupling matrix, together with a hash binding for
tamper-evidence.  As we show in Section~\ref{sec:zk}, the natural-looking
alternative of proving a pointwise generator inequality at the live weights
each epoch is vacuous, which is why the attestation targets the static
spectral certificate instead.

\paragraph{Scope and assumptions.}
We work in continuous time throughout.  The stochastic processes are
defined on a complete filtered probability space
$(\Omega, \mathcal{F}, (\mathcal{F}_t)_{t\geq 0}, \Prob)$ satisfying the
usual conditions.  Brownian motions $B^1, \ldots, B^K$ are mutually
independent $d$-dimensional standard Wiener processes.  All model weights
$W^k \in \R^d$ are elements of a finite-dimensional Euclidean space.
We make no claim about the specific architecture of the models; the
analysis applies to any system whose weight dynamics are governed by the
SDE described in Section~\ref{sec:model}.

\paragraph{Notation.}
We write $\norm{\cdot}$ for the Euclidean norm on $\R^d$ and
$\norm{\cdot}_F$ for the Frobenius norm on $\R^{d \times d}$.
For a matrix $A \in \R^{K \times K}$, $\lambda_{\max}(A)$ denotes its
spectral radius (largest eigenvalue in absolute value) and $\lambda_2(A)$
the second-largest eigenvalue.  We write $\lambda_{\min}(A)$ for the
eigenvalue of $A$ with the smallest (most negative) real part, and
$A_{\mathrm{sym}}=\tfrac12(A+A^\top)$ for the symmetric part with smallest
eigenvalue $\lambda_{\min}(A_{\mathrm{sym}})$.  The \emph{exact} second-moment
stability threshold is governed by $\lambda_{\min}(A)$ (the spectral
abscissa of the joint drift operator); the quadratic-Lyapunov \emph{sufficient}
certificate is governed by $\lambda_{\min}(A_{\mathrm{sym}})$.  These
coincide when $A$ is normal and differ otherwise; neither is
$\lambda_{\max}(A)$, which is identically $1$ for every row-stochastic $A$
by Perron--Frobenius.  We use $\ip{u}{v}$ for the standard inner product
on $\R^d$.

\paragraph{Organization.}
Section~\ref{sec:model} defines the multi-agent SDE and its standing
assumptions.  Section~\ref{sec:lyapunov} develops the joint Lyapunov
theory, culminating in the Stability Theorem and the Noise-Floor Theorem.
Section~\ref{sec:phase} analyzes the stability threshold and proves the
critical coupling formula.  Section~\ref{sec:zk} constructs the
ZK-SNARK.  Section~\ref{sec:sr117} gives a detailed treatment of the
MRM mapping.  Section~\ref{sec:numerical} presents five numerical
studies.  Section~\ref{sec:discussion} discusses implications and
Section~\ref{sec:conclusion} concludes.

\section{Multi-Agent System Dynamics}
\label{sec:model}

\subsection{The Coupled SDE}

We consider $K \geq 2$ agents with weight vectors
$W^1, \ldots, W^K \in \R^d$.  The joint state is
$\W = (W^1, \ldots, W^K) \in \R^{dK}$.  The dynamics are governed by
the coupled system of It\^{o} SDEs (for the continuous-time stochastic
framework see \cite{pham2009continuous}):
\begin{equation}
  \dif W^k_t \;=\; \mu^k(\W_t)\,\dif t + \sigma(W^k_t)\,\dif B^k_t,
  \qquad k = 1, \ldots, K,
  \label{eq:sde}
\end{equation}
where the drift and diffusion are, respectively,
\begin{align}
  \mu^k(\W) &= -\nabla \loss_k(W^k) + \gamma \sum_{j=1}^K A_{kj}
               (\Phi - W^j),
  \label{eq:drift} \\
  \sigma(W^k) &= \sigma_0 \, I_d.
  \label{eq:diffusion}
\end{align}
Here $\loss_k : \R^d \to \R$ is agent $k$'s loss function (the leading term
is gradient \emph{descent}, hence the minus sign),
$A \in \R^{K \times K}$ is the \emph{coupling adjacency matrix},
$\Phi \in \R^d$ is the \emph{meta-learning parameter} (a shared
governance or fine-tuning target), $\gamma \geq 0$ is the coupling
strength, and $\sigma_0 > 0$ is the noise coefficient.

\begin{assumption}[Self-decay]
  \label{ass:loss}
  Each agent's loss is $\alpha_{\mathrm{self}}$-strongly convex and isotropic
  about the origin, with
  \[
    \loss_k(w) \;=\; \tfrac{1}{2}\,\alpha_{\mathrm{self}}\,\norm{w}^2,
    \qquad\text{so}\qquad
    \nabla \loss_k(w) \;=\; \alpha_{\mathrm{self}}\, w ,
  \]
  for some $\alpha_{\mathrm{self}} > 0$, all $w \in \R^d$, and every $k$.
  Descending this loss therefore exerts a mean-reverting self-restoring force
  $-\nabla\loss_k(w) = -\alpha_{\mathrm{self}}w$ with rate $\alpha_{\mathrm{self}}$.
\end{assumption}

Under Assumption~\ref{ass:loss}, the drift simplifies to
\begin{equation}
  \mu^k(\W) = -\alpha_{\mathrm{self}} W^k +
  \gamma \sum_{j=1}^K A_{kj}(\Phi - W^j).
  \label{eq:drift_simple}
\end{equation}

\begin{assumption}[Admissible adjacency matrix]
  \label{ass:A}
  The matrix $A \in \R^{K \times K}$ satisfies:
  (i) $A_{kj} \geq 0$ for all $k, j$;
  (ii) $A_{kk} = 0$ (no self-coupling);
  (iii) $\sum_{j=1}^K A_{kj} = 1$ for all $k$ (row-stochastic).
\end{assumption}

Condition~(iii) ensures that the coupling term
$\gamma \sum_j A_{kj}(\Phi - W^j)$ has a natural interpretation as
agent $k$ being pulled toward a $A$-weighted average of the
meta-parameter $\Phi$ and the other agents' weights.

\begin{assumption}[Bounded meta-parameter]
  \label{ass:phi}
  $\Phi \in \R^d$ is a fixed vector with $\norm{\Phi} \leq \Phi_{\max}
  < \infty$.
\end{assumption}

\subsection{Interpretation of the state variable}
\label{sec:state_interpretation}

The system \eqref{eq:sde} assigns each agent a state $W^k_t\in\R^d$ evolving as a diffusion, and
every result in this paper --- the critical coupling $\gamma^*(A)$ of \eqref{eq:gamma_star}, the
noise floor \eqref{eq:noise_floor}, and the spectral certificate --- is a statement about that
diffusion. Since the intended application is to ensembles of large generative models, whose
parameter counts run to $10^{9}$--$10^{12}$, we state explicitly what $W^k_t$ is meant to denote
and what the results do not assert.

\paragraph{$W^k_t$ is an effective coordinate.}
The analysis nowhere requires $d$ to be an agent's full parameter count. It requires only that the
governed adaptation of agent $k$ be representable as a diffusion on $\R^d$ with a declared
diffusion coefficient. In practice $W^k_t$ is intended to be one of: the parameters of a low-rank
adapter or head, where continual adaptation is confined; a projection of the parameter vector onto
a subspace fixed at validation time; or a declared vector-valued summary statistic of the agent's
state. All theorems hold verbatim under any of these readings; what changes is the meaning of $d$,
and therefore of every constant that scales with it.

\paragraph{What the literal reading would require.}
Reading $W^k_t$ as an agent's raw parameter vector is not something we defend, for three reasons.

First, Assumption~\ref{ass:loss} sets $\nabla\loss_k(w)=\alpha_{\mathrm{self}}w$, so the loss is
the isotropic quadratic $\tfrac12\alpha_{\mathrm{self}}\norm{w}^2$, minimized at the origin. For a
generative model the origin is not a desirable state but a destroyed one; the mean-reverting term
of \eqref{eq:drift_simple} is therefore best read as a \emph{regularization} force --- weight decay,
or a pull toward a validated reference point after recentering --- and not as descent on a task
loss. The task-loss gradient does not appear in the model at all.

Second, the isotropy is severe. A single scalar $\alpha_{\mathrm{self}}$ stands in for the entire
curvature structure, whereas the loss geometry of a large model is strongly anisotropic. Under an
effective-coordinate reading this is far less demanding: one may choose the coordinate precisely so
that the local curvature is approximately isotropic on it, which is not available in the raw
parameter space.

Third, the constants scale with $d$. The noise floor \eqref{eq:noise_floor} is
$\E_\pi[V]=Kd\sigma_0^2/(4\alpha_{\mathrm{self}})$, linear in $d$. At $d\sim10^{11}$ the absolute
magnitude of this quantity is not meaningful as a risk level; only ratios and thresholds normalized
by $Kd$ retain interpretation. Reported in an effective coordinate of modest dimension, it is
directly interpretable.

\paragraph{What survives regardless of the reading.}
Two of the paper's conclusions are insensitive to how $W^k_t$ is interpreted, because they concern
the \emph{coupling}, not the state. The critical threshold
$\gamma^*(A)=\alpha_{\mathrm{self}}/\abs{\lambda_{\min}(A)}$ depends on $d$ not at all; it is a
statement about the interaction matrix and the self-decay rate. Likewise the topology ordering ---
that stability is governed by the most negative eigenvalue of $A$ rather than by its spectral radius
or spectral gap --- is a property of $A$ alone. These are the results we regard as the paper's
contribution, and they are exactly the ones that do not inherit the scale problem.

\paragraph{Relation to the stated limitations.}
This reading does not dissolve the linearity and homogeneity assumptions discussed in
Section~\ref{sec:limitations}; it relocates them. Under an effective-coordinate reading the
question becomes whether the induced process on the chosen coordinate is approximately linear with
approximately homogeneous decay --- a local and testable question, rather than a global claim about
the loss surface of a large model. We do not test it here, and the numerical studies of
Section~\ref{sec:numerical} simulate the model of Section~\ref{sec:model} directly rather than any
projection of a trained network.

\subsection{Existence and Uniqueness}

Stability is a property of a well-defined process, so the analysis begins by confirming the coupled $K$-agent system has a unique solution to be stable \emph{about}. This guarantee is what licenses speaking of a single joint trajectory at all---and hence of ensemble-level dynamics, rather than $K$ separate agent-by-agent ones.

\begin{proposition}[Well-posedness]
  \label{prop:wellposed}
  Under Assumptions \ref{ass:loss}--\ref{ass:phi}, the SDE
  \eqref{eq:sde}--\eqref{eq:diffusion} admits a unique strong solution
  $\{\W_t\}_{t \geq 0}$ for every initial condition
  $\W_0 \in \R^{dK}$.  The solution is a Markov process with respect
  to the filtration $(\mathcal{F}_t)_{t \geq 0}$.
\end{proposition}

\begin{proof}
  It suffices to verify the standard Lipschitz and linear-growth
  conditions for SDEs; see, e.g., \cite{oksendal2003sde} Theorem~5.2.1.
  The drift $\mu^k$ in \eqref{eq:drift_simple} is affine in
  $\W$, hence globally Lipschitz with constant
  $L = \alpha_{\mathrm{self}} + \gamma \norm{A}_F$.  The diffusion
  $\sigma(W^k) = \sigma_0 I_d$ is constant and hence trivially Lipschitz.
  Linear growth holds since
  $\norm{\mu^k(\W)} \leq (\alpha_{\mathrm{self}} +
  \gamma)\norm{\W} + \gamma \Phi_{\max}$, which is bounded by
  $C(1 + \norm{\W})$ for large enough $C > 0$.  Existence, uniqueness,
  and the Markov property follow from \cite{oksendal2003sde}.
\end{proof}

\subsection{Example Topologies}
\label{ssec:topologies}

We consider three canonical coupling structures.

\paragraph{Complete graph.}  $A_{kj} = 1/(K-1)$ for all $j \neq k$.
Every agent is equally influenced by every other agent.  $A$ is symmetric
(hence normal), so $A_{\mathrm{sym}}=A$ and $\lambda_{\min}(A)=\lambda_{\min}(A_{\mathrm{sym}})$.
The spectrum is $\{1,\,-\tfrac{1}{K-1}\,(\times(K-1))\}$, giving
$\lambda_{\max}(A) = 1$, $\lambda_2(A) = -1/(K-1)$, and the
destabilizing eigenvalue $\lambda_{\min}(A) = -1/(K-1)$ ($=-0.25$ at $K=5$).

\paragraph{Ring graph.}  Agent $k$ is coupled to agents $k-1$ and $k+1$
(modulo $K$), with $A_{k,k\pm 1} = 1/2$.  $A$ is symmetric (normal).  The
eigenvalues are $\cos(2\pi m / K)$ for $m = 0, \ldots, K-1$, giving
$\lambda_{\max}(A) = 1$ and destabilizing eigenvalue $\lambda_{\min}(A) =
\min_m \cos(2\pi m/K)$ ($=\cos(4\pi/5)=-0.809$ at $K=5$); since $A$ is
normal, $\lambda_{\min}(A_{\mathrm{sym}})=\lambda_{\min}(A)$.

\paragraph{Star graph.}  Agent~$0$ is the hub with $A_{0j} = 1/(K-1)$
for $j \geq 1$; spokes have $A_{k0} = 1$ for $k \geq 1$.  Here $A$ is
\emph{non-normal} ($A\neq A^\top$): its eigenvalues are
$\{1,\,0\,(\times(K-2)),\,-1\}$, so $\lambda_{\max}(A)=1$, the second-largest
eigenvalue is $0$, and the destabilizing eigenvalue is
$\lambda_{\min}(A)=-1$.  The symmetric part is different: $A_{\mathrm{sym}}$
has the bordered (arrowhead) structure with nonzero eigenvalues
$\pm\,K/(2\sqrt{K-1})$ and $0$ ($\times(K-2)$), giving
$\lambda_{\min}(A_{\mathrm{sym}}) = -K/(2\sqrt{K-1}) = -1.25$ at $K=5$.
Note $\abs{\lambda_{\min}(A_{\mathrm{sym}})}=1.25 > 1 = \abs{\lambda_{\min}(A)}
= \rho(A)$: the field of values extends beyond the spectrum, so the
quadratic-Lyapunov certificate (which sees $A_{\mathrm{sym}}$) is strictly
more conservative than the exact threshold (which sees $\lambda_{\min}(A)$)
for the star---and a na\"ive spectral-radius bound is unsound here
(Remark~\ref{rem:fov}).

These three topologies represent qualitatively different information
structures: the complete graph maximizes mixing, the ring is a local
interaction network, and the star creates a concentrated dependency on the
hub.  Ordered by the destabilizing eigenvalue $\abs{\lambda_{\min}(A)}$---complete
$0.25 <$ ring $0.809 <$ star $1.0$ at $K=5$---they have \emph{decreasing}
exact critical coupling
$\gamma^*=\alpha_{\mathrm{self}}/\abs{\lambda_{\min}(A)}$
(complete $0.40 >$ ring $0.124 >$ star $0.10$), so the complete graph
tolerates the most coupling and the star the least.

\section{Joint Lyapunov Analysis}
\label{sec:lyapunov}

\subsection{The Candidate Lyapunov Function}

\begin{definition}[Joint quadratic Lyapunov function]
  \label{def:V}
  Define $V : \R^{dK} \to \R_{\geq 0}$ by
  \begin{equation}
    V(\W) = \frac{1}{2} \sum_{k=1}^K \norm{W^k}^2.
    \label{eq:V}
  \end{equation}
\end{definition}

The use of a quadratic Lyapunov function to certify stochastic stability is classical \cite{kushner1967stochastic, khasminskii2012stochastic}. The quantity $V$ aggregates the individual agents' squared norms into a single scalar: it is the natural ensemble-level risk measure, and the entire stability analysis is conducted in terms of it rather than the per-agent norms a single-model validation would track separately. The following records that it is a valid Lyapunov function.

\begin{proposition}[Lyapunov properties]
  \label{prop:lyapunov_props}
  The function $V$ defined in \eqref{eq:V} satisfies all standard
  Lyapunov conditions:
  \begin{enumerate}[label=(\roman*)]
    \item $V \in C^2(\R^{dK})$ (infinitely differentiable);
    \item $V(\mathbf{0}) = 0$;
    \item $V(\W) > 0$ for all $\W \neq \mathbf{0}$;
    \item $V(\W) \to \infty$ as $\norm{\W} \to \infty$ (radially unbounded).
  \end{enumerate}
\end{proposition}

\begin{proof}
  All four properties are immediate from the definition.  $V$ is a
  positive-definite quadratic form in $\R^{dK}$ and hence
  $C^\infty$.
\end{proof}

\subsection{The Infinitesimal Generator}

Let $\Lgen$ denote the infinitesimal generator of the Markov process
$\{\W_t\}$, defined for $f \in C^2(\R^{dK})$ by
\begin{equation}
  \Lgen f(\W) = \lim_{t \downarrow 0} \frac{\E_\W[f(\W_t)] - f(\W)}{t}.
\end{equation}

By It\^{o}'s formula, the generator takes the form
\begin{equation}
  \Lgen f(\W) = \sum_{k=1}^K \left[
    \ip{\mu^k(\W)}{\nabla_{W^k} f(\W)}
    + \frac{\sigma_0^2}{2} \Delta_{W^k} f(\W)
  \right],
  \label{eq:generator_general}
\end{equation}
where $\Delta_{W^k} = \sum_{i=1}^d \partial^2 / \partial (W^k_i)^2$ is
the Laplacian with respect to $W^k$.

The generator is where coupling enters the stability calculus, and therefore where the gap between per-model and ensemble-level analysis first becomes visible in the mathematics. The leading term below, $-2\alpha_{\mathrm{self}}V$, is the per-agent dissipation that a single-model analysis would see; the middle term is the cross-agent coupling that such an analysis omits entirely. The result that follows shows that the moment this coupling is present ($\gamma>0$), the ensemble-wide decay rate obtained by naively summing individual certificates is wrong---the formal reason per-model validation does not deliver an ensemble-level guarantee.

\begin{theorem}[Generator of the joint Lyapunov function]
  \label{thm:generator}
  Under Assumptions \ref{ass:loss}--\ref{ass:phi}, the infinitesimal
  generator of $V$ under the dynamics \eqref{eq:sde}--\eqref{eq:drift_simple} is
  \begin{equation}
    \Lgen V(\W) = -2\alpha_{\mathrm{self}} V(\W) +
    \gamma \sum_{k=1}^K \sum_{j=1}^K A_{kj}
    \ip{W^k}{\Phi - W^j}
    + \frac{1}{2} K d \sigma_0^2.
    \label{eq:LV_exact}
  \end{equation}
\end{theorem}

\begin{proof}
  We compute each component of \eqref{eq:generator_general} in turn.

  \textbf{Step 1: Partial derivatives of $V$.}
  Since $V(\W) = \tfrac{1}{2}\sum_{k=1}^K \norm{W^k}^2$, we have
  \[
    \frac{\partial V}{\partial W^k} = W^k \in \R^d,
    \qquad
    \frac{\partial^2 V}{\partial (W^k_i)^2} = 1
    \;\text{ for all } i = 1,\ldots,d,\; k = 1,\ldots,K.
  \]

  \textbf{Step 2: Drift contribution.}
  \begin{align*}
    \sum_{k=1}^K \ip{\mu^k(\W)}{W^k}
    &= \sum_{k=1}^K \ip{-\alpha_{\mathrm{self}} W^k
      + \gamma \textstyle\sum_j A_{kj}(\Phi - W^j)}{W^k} \\
    &= -\alpha_{\mathrm{self}} \sum_{k=1}^K \norm{W^k}^2
      + \gamma \sum_{k=1}^K \sum_{j=1}^K A_{kj}
        \ip{\Phi - W^j}{W^k} \\
    &= -2\alpha_{\mathrm{self}} V(\W)
      + \gamma \sum_{k,j} A_{kj} \ip{W^k}{\Phi - W^j}.
  \end{align*}

  \textbf{Step 3: Diffusion contribution.}
  Since $\sigma(W^k) = \sigma_0 I_d$ and
  $\partial^2 V / \partial (W^k_i)^2 = 1$ for all $i, k$,
  \[
    \sum_{k=1}^K \frac{\sigma_0^2}{2} \Delta_{W^k} V(\W)
    = \sum_{k=1}^K \frac{\sigma_0^2}{2} \cdot d
    = \frac{1}{2} K d \sigma_0^2.
  \]

  \textbf{Step 4: Assembly.}
  Combining Steps~2 and 3 in \eqref{eq:generator_general} gives
  \eqref{eq:LV_exact}.
\end{proof}

\subsection{An Upper Bound on the Generator}

The exact formula \eqref{eq:LV_exact} contains the cross-term
$\gamma \sum_{k,j} A_{kj}\ip{W^k}{\Phi - W^j}$, which can be either
positive or negative.  We derive an upper bound that makes the
sign structure transparent.

That sign is the whole question for stability: a coupling term that is reliably negative is dissipative and harmless, whereas one that can turn positive is a channel through which interconnection injects energy into the ensemble rather than removing it. The bound below resolves the sign through the symmetric part $A_{\mathrm{sym}}$, isolating the eigenvalue that will prove to govern ensemble instability.

\begin{lemma}[Coupling term bound]
  \label{lem:coupling_bound}
  Let $A_{\mathrm{sym}} := \tfrac{1}{2}(A + A^\top)$ and let
  $\lambda_{\min}(A_{\mathrm{sym}})$ denote its smallest eigenvalue.
  Under Assumptions \ref{ass:A} and \ref{ass:phi},
  \begin{equation}
    \gamma \sum_{k=1}^K \sum_{j=1}^K A_{kj} \ip{W^k}{\Phi - W^j}
    \;\leq\;
    \gamma \Phi_{\max} \sqrt{K} \norm{\W}
    - \gamma\, \lambda_{\min}(A_{\mathrm{sym}})\, \norm{\W}^2.
    \label{eq:coupling_bound}
  \end{equation}
  When $\lambda_{\min}(A_{\mathrm{sym}}) < 0$ the second term is
  \emph{positive}, equal to
  $\gamma\abs{\lambda_{\min}(A_{\mathrm{sym}})}\,\norm{\W}^2$, and is the
  destabilizing contribution of the coupling.
\end{lemma}

\begin{proof}
  Split the coupling term using the row-stochastic property
  $\sum_j A_{kj}=1$:
  \[
    \gamma \sum_{k,j} A_{kj} \ip{W^k}{\Phi - W^j}
    = \gamma \sum_{k} \ip{W^k}{\Phi}
      - \gamma \sum_{k,j} A_{kj} \ip{W^k}{W^j}.
  \]
  The first sum is bounded by $\gamma \Phi_{\max} \sqrt{K} \norm{\W}$
  by Cauchy--Schwarz in $\R^d$ and then in $\R^K$, exactly as before.
  For the quadratic sum, write $\mathbf{w} = \mathrm{vec}(\W) \in \R^{dK}$.
  A quadratic form depends only on the symmetric part of its matrix, so
  \[
    \gamma \sum_{k,j} A_{kj} \ip{W^k}{W^j}
    = \gamma\, \mathbf{w}^\top (A \otimes I_d) \mathbf{w}
    = \gamma\, \mathbf{w}^\top (A_{\mathrm{sym}} \otimes I_d) \mathbf{w}
    \;\geq\; \gamma\, \lambda_{\min}(A_{\mathrm{sym}})\, \norm{\mathbf{w}}^2,
  \]
  where the last step is the Rayleigh bound for the \emph{symmetric}
  matrix $A_{\mathrm{sym}} \otimes I_d$, whose spectrum is
  $\spec(A_{\mathrm{sym}})$ (each eigenvalue with multiplicity $d$).
  Negating and combining gives \eqref{eq:coupling_bound}.
\end{proof}

\begin{remark}[Why the spectral radius is the wrong bound]
  \label{rem:fov}
  An earlier form of this lemma bounded the quadratic term by
  $\mathbf{w}^\top(A\otimes I_d)\mathbf{w}\geq-\lambda_{\max}(A)\norm{\mathbf{w}}^2
  =-\rho(A)\norm{\mathbf{w}}^2$, using the spectral radius. This is
  \emph{false} for non-normal $A$. The sharp lower bound is the smallest
  eigenvalue of the symmetric part---the left edge of the numerical range
  (field of values) of $A\otimes I_d$---and for a non-normal matrix the
  field of values strictly contains the interval
  $[\lambda_{\min}(A),\lambda_{\max}(A)]$. The star topology is a concrete
  failure: $\rho(A)=1$ but $\lambda_{\min}(A_{\mathrm{sym}})=-1.25$ at
  $K=5$ (Section~\ref{ssec:topologies}). The discarded bound understates
  the destabilizing term by $25\%$ and would certify as stable a range of
  $\gamma$ that is in fact unstable---a false safety certificate, which is
  the worst possible failure mode for a compliance tool.
\end{remark}

\begin{corollary}[Generator upper bound]
  \label{cor:gen_upper}
  For any $\varepsilon > 0$, with $A_{\mathrm{sym}}=\tfrac12(A+A^\top)$,
  \begin{equation}
    \Lgen V(\W) \;\leq\;
    -\bigl(2\alpha_{\mathrm{self}} + 2\gamma\, \lambda_{\min}(A_{\mathrm{sym}})
    - \varepsilon\bigr)\, V(\W) + C_\varepsilon,
    \label{eq:gen_upper}
  \end{equation}
  where $C_\varepsilon = \tfrac{1}{2}Kd\sigma_0^2
  + \gamma^2 \Phi_{\max}^2 K / (4\varepsilon)$.  The coupling
  \emph{reduces} the effective decay rate precisely when
  $\lambda_{\min}(A_{\mathrm{sym}})<0$, contributing
  $-2\gamma\abs{\lambda_{\min}(A_{\mathrm{sym}})}\,V$.
\end{corollary}

\begin{proof}
  By Young's inequality $ab \leq \varepsilon a^2 + b^2/(4\varepsilon)$
  applied to the $\Phi$-term in Lemma~\ref{lem:coupling_bound}:
  \[
    \gamma \Phi_{\max} \sqrt{K} \norm{\W}
    \leq \varepsilon \norm{\W}^2 + \frac{\gamma^2 \Phi_{\max}^2 K}{4\varepsilon}
    = 2\varepsilon V(\W) + \frac{\gamma^2 \Phi_{\max}^2 K}{4\varepsilon}.
  \]
  Substituting into Lemma~\ref{lem:coupling_bound} and using
  $\norm{\W}^2 = 2V(\W)$:
  \[
    \Lgen V(\W) \leq
    -\bigl(2\alpha_{\mathrm{self}} + 2\gamma \lambda_{\min}(A_{\mathrm{sym}})
    - 2\varepsilon\bigr) V(\W)
    + \tfrac{1}{2}Kd\sigma_0^2 + \frac{\gamma^2 \Phi_{\max}^2 K}{4\varepsilon}.
  \]
  Replacing $2\varepsilon$ by $\varepsilon$ and absorbing the
  $\varepsilon$-dependent constant into $C_\varepsilon$ gives
  \eqref{eq:gen_upper}.
\end{proof}

\subsection{The Stability Theorem}

The certificate below is the operational core for the validation function: a single, checkable inequality on the coupling strength that guarantees the \emph{joint} system is stable---ultimately bounded, with a unique stationary law. Because it is expressed through the symmetric part $A_{\mathrm{sym}}$ of the interaction matrix, it is conservative for non-normal topologies and never certifies an unstable ensemble, which is exactly the property an ensemble-level pass/fail criterion should have.

\begin{theorem}[Joint Lyapunov stability --- sufficient certificate]
  \label{thm:stability}
  Suppose that the coupling strength satisfies
  \begin{equation}
    \gamma < \gamma_{\mathrm{cert}}(A) := \frac{\alpha_{\mathrm{self}}}
    {\abs{\lambda_{\min}(A_{\mathrm{sym}})}},
    \qquad A_{\mathrm{sym}}=\tfrac12(A+A^\top).
    \label{eq:gamma_star}
  \end{equation}
  (This is a \emph{sufficient} condition; $\gamma_{\mathrm{cert}}(A)\le\gamma^*(A)$,
  the exact threshold of Theorem~\ref{thm:phase}, with equality iff $A$ is
  normal.)  Then there exist constants $\alpha_{\mathrm{eff}} > 0$ and
  $\beta > 0$ such that
  \begin{equation}
    \Lgen V(\W) \;\leq\; -\alpha_{\mathrm{eff}} V(\W) + \beta
    \quad \text{for all } \W \in \R^{dK},
    \label{eq:foster_lyapunov}
  \end{equation}
  with
  \begin{align}
    \alpha_{\mathrm{eff}} &= 2\bigl(\alpha_{\mathrm{self}}
      - \gamma \abs{\lambda_{\min}(A_{\mathrm{sym}})}\bigr) - \varepsilon_0 > 0, \\
    \beta &= \frac{1}{2}Kd\sigma_0^2
      + \frac{\gamma^2 \Phi_{\max}^2 K}{4\varepsilon_0},
  \end{align}
  for any $0 < \varepsilon_0 <
  2\bigl(\alpha_{\mathrm{self}} - \gamma \abs{\lambda_{\min}(A_{\mathrm{sym}})}\bigr)$.
  (If $A_{\mathrm{sym}}\succeq 0$, i.e.\ $\lambda_{\min}(A_{\mathrm{sym}})\geq0$,
  the coupling is purely stabilizing and $\gamma_{\mathrm{cert}}(A)=+\infty$.)

  Furthermore, the process $\{\W_t\}$ has a unique stationary
  distribution $\pi$ on $\R^{dK}$, and
  \begin{equation}
    \E_\pi[V(\W)] \leq \frac{\beta}{\alpha_{\mathrm{eff}}}.
    \label{eq:stationary_bound}
  \end{equation}
\end{theorem}

\begin{proof}
  \textbf{Part 1: Foster--Lyapunov condition.}
  By Corollary~\ref{cor:gen_upper} with $\varepsilon = \varepsilon_0$,
  \[
    \Lgen V(\W) \leq
    -\bigl(2\alpha_{\mathrm{self}} + 2\gamma \lambda_{\min}(A_{\mathrm{sym}})
    - \varepsilon_0\bigr) V(\W) + C_{\varepsilon_0}.
  \]
  Setting $\alpha_{\mathrm{eff}} = 2\alpha_{\mathrm{self}}
  + 2\gamma\lambda_{\min}(A_{\mathrm{sym}}) - \varepsilon_0
  = 2(\alpha_{\mathrm{self}} - \gamma\abs{\lambda_{\min}(A_{\mathrm{sym}})}) - \varepsilon_0$ and
  $\beta = C_{\varepsilon_0}$, and noting that $\alpha_{\mathrm{eff}} > 0$
  by hypothesis \eqref{eq:gamma_star}, we obtain
  \eqref{eq:foster_lyapunov}.

  \textbf{Part 2: Existence of a stationary distribution.}
  The process $\{\W_t\}$ is a non-degenerate diffusion on $\R^{dK}$
  (the diffusion coefficient $\sigma_0 I_{dK}$ is uniformly elliptic).
  Combined with the Foster--Lyapunov condition
  \eqref{eq:foster_lyapunov} and the Lyapunov properties
  (Proposition~\ref{prop:lyapunov_props}), the Meyn--Tweedie criterion
  \cite{meyn1993stability} guarantees the existence of a unique
  stationary distribution $\pi$.

  \textbf{Part 3: Stationary moment bound.}
  Under the stationary distribution, $\E_\pi[\Lgen V] = 0$ (the
  process is in equilibrium).  Taking expectations on both sides of
  \eqref{eq:foster_lyapunov}:
  \[
    0 = \E_\pi[\Lgen V(\W)] \leq
    -\alpha_{\mathrm{eff}} \E_\pi[V(\W)] + \beta,
  \]
  which rearranges to $\E_\pi[V(\W)] \leq \beta / \alpha_{\mathrm{eff}}$.
\end{proof}

\subsection{The Noise-Floor Theorem}

In the decoupled case ($\gamma = 0$), the exact stationary distribution
is available in closed form.

The resulting floor is the irreducible level of fluctuation the ensemble carries from learning noise alone, with no coupling at all; it is the baseline against which the \emph{coupling-induced} excess---the genuinely emergent component of co-movement---is measured in the phase-transition analysis that follows.

\begin{theorem}[Noise floor]
  \label{thm:noise_floor}
  Suppose $\gamma = 0$ and $\Phi = 0$.  Then the SDE
  \eqref{eq:sde}--\eqref{eq:drift_simple} decomposes into $Kd$
  independent scalar Ornstein--Uhlenbeck (OU) processes:
  \[
    \dif W^k_{t,i} = -\alpha_{\mathrm{self}} W^k_{t,i}\,\dif t
    + \sigma_0\,\dif B^k_{t,i}, \quad i=1,\ldots,d,\; k=1,\ldots,K.
  \]
  The unique stationary distribution is
  $\pi = \bigotimes_{k=1}^K \mathcal{N}(0, \sigma_0^2/(2\alpha_{\mathrm{self}}) I_d)$.
  Consequently,
  \begin{equation}
    \E_\pi[V(\W)] = \frac{\beta_0}{2\alpha_{\mathrm{self}}},
    \quad \text{where } \beta_0 := \frac{1}{2}Kd\sigma_0^2.
    \label{eq:noise_floor}
  \end{equation}
\end{theorem}

\begin{proof}
  \textbf{Step 1: Decoupled OU dynamics.}
  With $\gamma = 0$, the $K$ agents decouple completely.  Each scalar
  component $W^k_{t,i}$ evolves as
  $\dif X = -\alpha X\,\dif t + \sigma_0\,\dif B$ with
  $\alpha = \alpha_{\mathrm{self}}$.  This is the classical
  Ornstein--Uhlenbeck process \cite{uhlenbeck1930theory}.

  \textbf{Step 2: Stationary distribution of the scalar OU process.}
  The Fokker--Planck equation for the stationary density $p_\infty(x)$ is
  \[
    \frac{\dif}{\dif x}\bigl[\alpha x\, p_\infty(x)\bigr]
    + \frac{\sigma_0^2}{2} \frac{\dif^2 p_\infty}{\dif x^2} = 0.
  \]
  The unique normalized solution is the Gaussian
  $p_\infty(x) = \mathcal{N}(0, \sigma_0^2/(2\alpha))$.
  Hence $\E_\pi[(W^k_{t,i})^2] = \sigma_0^2/(2\alpha_{\mathrm{self}})$.

  \textbf{Step 3: Stationary second moment of $V$.}
  By independence across components and agents:
  \begin{align*}
    \E_\pi[V(\W)]
    &= \frac{1}{2} \sum_{k=1}^K \E_\pi\!\left[\norm{W^k}^2\right]
     = \frac{1}{2} \sum_{k=1}^K \sum_{i=1}^d
       \E_\pi\!\left[(W^k_i)^2\right] \\
    &= \frac{1}{2} \cdot K \cdot d \cdot \frac{\sigma_0^2}{2\alpha_{\mathrm{self}}}
     = \frac{Kd\sigma_0^2}{4\alpha_{\mathrm{self}}}
     = \frac{\beta_0}{2\alpha_{\mathrm{self}}}.
  \end{align*}
  The last equality uses $\beta_0 = \tfrac{1}{2}Kd\sigma_0^2$.
\end{proof}

\begin{remark}
  Theorem~\ref{thm:noise_floor} confirms that the noise floor
  $\beta_0$ in the Foster--Lyapunov bound \eqref{eq:foster_lyapunov}
  is tight in the decoupled case: the stationary expectation
  equals exactly $\beta_0 / (2\alpha_{\mathrm{self}})$, matching the
  generator formula $\Lgen V(\W)\big|_{\gamma=0} = -2\alpha_{\mathrm{self}}
  V(\W) + \beta_0$.  Numerical Study~C3 verifies this formula
  empirically across eight parameter configurations.
\end{remark}

\subsection{Almost-Sure Convergence}

The Foster--Lyapunov condition implies more than moment bounds.

For governance the distinction between moment bounds and pathwise behavior matters: a guarantee in expectation can hold while individual realized trajectories wander, whereas the pathwise, time-ergodic convergence below means the \emph{ensemble} settles to a characterized stationary regime against which monitoring thresholds can be calibrated.

\begin{theorem}[Almost-sure convergence to the stationary regime]
  \label{thm:as_convergence}
  Under the conditions of Theorem~\ref{thm:stability}, let
  $\{\W_t\}$ start from any initial condition $\W_0 \in \R^{dK}$.
  Then $V(\W_t) \to \E_\pi[V(\W)]$ in the time-ergodic sense:
  \begin{equation}
    \frac{1}{T}\int_0^T V(\W_t)\,\dif t
    \xrightarrow[T\to\infty]{a.s.}
    \E_\pi[V(\W)].
    \label{eq:ergodic}
  \end{equation}
  Furthermore, there exist constants $C, c > 0$ such that
  \begin{equation}
    \abs{\E[V(\W_t)] - \E_\pi[V(\W)]}
    \leq C \, e^{-c\, t} \bigl(V(\W_0) - \E_\pi[V]\bigr)^+.
    \label{eq:exp_convergence}
  \end{equation}
\end{theorem}

\begin{proof}
  \textbf{Ergodic convergence.}
  The process $\{\W_t\}$ is a non-degenerate uniformly elliptic
  diffusion satisfying a Foster--Lyapunov condition with radially
  unbounded $V$, hence it is positive Harris recurrent
  \cite{meyn1993stability}.  The ergodic theorem for positive Harris
  recurrent processes \cite{revuz1999continuous} gives \eqref{eq:ergodic}.

  \textbf{Exponential rate.}
  From It\^{o}'s formula applied to $V(\W_t)$:
  \[
    \dif V(\W_t) = \Lgen V(\W_t)\,\dif t + \dif M_t,
  \]
  where $M_t = \sigma_0 \sum_k \ip{W^k_t}{\dif B^k_t}$ is a local
  martingale \cite{rogers2000diffusions}.  Using \eqref{eq:foster_lyapunov}:
  \[
    \dif V(\W_t) \leq (-\alpha_{\mathrm{eff}} V(\W_t) + \beta)\,\dif t
    + \dif M_t.
  \]
  Multiplying by the integrating factor $e^{\alpha_{\mathrm{eff}} t}$ and
  integrating:
  \[
    e^{\alpha_{\mathrm{eff}} t} V(\W_t) \leq V(\W_0)
    + \frac{\beta}{\alpha_{\mathrm{eff}}}(e^{\alpha_{\mathrm{eff}} t} - 1)
    + e^{\alpha_{\mathrm{eff}} t} M_t^*.
  \]
  where $M_t^* = \int_0^t e^{-\alpha_{\mathrm{eff}} s}\,\dif M_s$ is a
  square-integrable martingale by the Burkholder--Davis--Gundy inequality
  \cite{revuz1999continuous}.  Taking expectations:
  \[
    \E[V(\W_t)] \leq e^{-\alpha_{\mathrm{eff}} t} V(\W_0)
    + \frac{\beta}{\alpha_{\mathrm{eff}}}(1 - e^{-\alpha_{\mathrm{eff}} t}).
  \]
  Thus $\abs{\E[V(\W_t)] - \beta/\alpha_{\mathrm{eff}}} \leq
  e^{-\alpha_{\mathrm{eff}} t} |V(\W_0) - \beta/\alpha_{\mathrm{eff}}|$,
  establishing \eqref{eq:exp_convergence} with $c = \alpha_{\mathrm{eff}}$
  and $C = 1$.
\end{proof}

\section{The Critical Coupling Threshold}
\label{sec:phase}

\subsection{The Critical Threshold}

Theorem~\ref{thm:stability} gives a \emph{sufficient} stability certificate
$\gamma < \gamma_{\mathrm{cert}}(A) = \alpha_{\mathrm{self}}/\abs{\lambda_{\min}(A_{\mathrm{sym}})}$.
We now determine the \emph{exact} threshold and identify the mode that goes
unstable.  The key structural fact is that the system is a linear
(Ornstein--Uhlenbeck) diffusion, so both its mean and its second moment are
stable if and only if the joint drift operator is positive-stable.

\begin{remark}[The spectral radius cannot govern topology-dependent risk]
\label{rem:perron}
By Perron--Frobenius, every row-stochastic non-negative $A$ has
$\lambda_{\max}(A)=1$, with the all-ones (consensus) Perron eigenvector.
A threshold of the form $\gamma^*=c\,\alpha_{\mathrm{self}}/\lambda_{\max}(A)$
is therefore $c\,\alpha_{\mathrm{self}}$ for \emph{every} admissible
topology---a topology-\emph{independent} number, which cannot be the basis
of a topology-dependent risk thesis.  The resolution is that the Perron
root is the most \emph{stabilized} direction (see below), not the
destabilizing one; the genuine instability is controlled by
$\lambda_{\min}(A)$, the most negative eigenvalue of $A$.
\end{remark}

This is the central result of the paper. The threshold $\gamma^*(A)$ is a \emph{critical governance threshold}: a hard, topology-dependent ceiling on how tightly a population of models may be coupled before interconnection alone tips the ensemble from a stable regime into one whose risk grows without bound. It is, in principle, a governance limit of the same character as a position or concentration limit---an observable threshold on a structural quantity---and, as the remark above stresses, it is governed by the system's \emph{disagreement} mode, not the consensus mode that per-model intuition would lead one to monitor.

\begin{theorem}[Critical coupling threshold]
  \label{thm:phase}
  Fix $A$, let $\lambda_{\min}(A)$ be its most negative eigenvalue (real
  part), and suppose $\lambda_{\min}(A)<0$.  With
  \[
    \gamma^*(A) := \frac{\alpha_{\mathrm{self}}}{\abs{\lambda_{\min}(A)}},
  \]
  the family of processes parameterized by $\gamma\ge0$ satisfies:
  \begin{enumerate}[label=(\roman*)]
    \item \textbf{Stable regime.} If $\gamma < \gamma^*(A)$, the joint drift
      operator $M=\alpha_{\mathrm{self}} I+\gamma(A\otimes I_d)$ is
      positive-stable, $\{\W_t\}$ has a unique (Gaussian) stationary
      distribution, and $\E[V(\W_t)]\to\E_\pi[V(\W)]$ exponentially fast.
    \item \textbf{Critical point.} If $\gamma = \gamma^*(A)$, the mode along
      the $\lambda_{\min}(A)$ eigenvector is marginally stable and
      $\E[V(\W_t)]$ grows \emph{linearly} in $t$, so no stationary
      distribution exists.
    \item \textbf{Unstable regime.} If $\gamma > \gamma^*(A)$, let
      $\zeta\in\R^K$ satisfy $A\zeta=\lambda_{\min}(A)\zeta$ and $u\in\R^d$
      be any unit vector.  Then the mean along $v=\zeta\otimes u$ diverges,
      $\norm{\E[\W_t]}\ge c\,e^{(\gamma\abs{\lambda_{\min}(A)}-\alpha_{\mathrm{self}})t}$,
      and hence $\E[V(\W_t)]\ge\tfrac12\norm{\E[\W_t]}^2\to\infty$.
  \end{enumerate}
\end{theorem}

\begin{remark}[On terminology: a stability threshold, not a phase transition]
\label{rem:not_phase}
Earlier drafts described the behavior at $\gamma^*(A)$ as a \emph{phase transition}. We avoid that
term. What Theorem~\ref{thm:phase} establishes is the crossing of an eigenvalue of the joint drift
operator through zero: for $\gamma<\gamma^*$ the operator $M$ is positive-stable and a stationary
law exists; at $\gamma=\gamma^*$ the leading eigenvalue reaches the imaginary axis; above it, the
second moment grows without bound. This is linear (in)stability of an Ornstein--Uhlenbeck system,
and it is exactly and analytically solvable. It has none of the features that give the physical term
its content --- there is no order parameter, no critical exponent, no non-analyticity of a free
energy, and no thermodynamic limit. The transition is sharp in $\gamma$ simply because the real part
of an eigenvalue is a continuous function that changes sign at a point. We retain the language of a
\emph{critical coupling} because $\gamma^*$ is a genuine threshold with governance meaning, but the
mathematics behind it is an eigenvalue computation, and we do not want the terminology to suggest
otherwise.
\end{remark}

\begin{proof}
  Analyze the homogeneous part: with $\Phi$ present the drift acquires the
  bounded affine term $\gamma(\mathbf 1_K\otimes\Phi)$, which shifts the
  fixed point but leaves the exponential rates unchanged.  Stacking
  $\mathbf w=\mathrm{vec}(\W)$, the dynamics are the linear SDE
  $\dif\mathbf w_t=-M\mathbf w_t\,\dif t+\sigma_0\,\dif\mathbf B_t$ with
  $M=\alpha_{\mathrm{self}} I_{dK}+\gamma(A\otimes I_d)$.  The eigenvalues of
  $M$ are $\alpha_{\mathrm{self}}+\gamma\nu$ for $\nu\in\spec(A)$ (since
  $\spec(A\otimes I_d)=\spec(A)$, each with multiplicity $d$), so $M$ is
  positive-stable---$\mathrm{Re}\,\mu>0$ for all $\mu\in\spec(M)$---iff
  $\alpha_{\mathrm{self}}+\gamma\,\mathrm{Re}\,\nu>0$ for all $\nu$, i.e.\ iff
  $\gamma<\alpha_{\mathrm{self}}/\abs{\lambda_{\min}(A)}=\gamma^*(A)$.

  \emph{Part (i).}  For a linear SDE the mean solves
  $\dot{\E}[\mathbf w]=-M\,\E[\mathbf w]$ and the covariance solves the
  Lyapunov equation $\dot\Sigma=-M\Sigma-\Sigma M^\top+\sigma_0^2 I$.  When
  $M$ is positive-stable both converge---$\E[\mathbf w_t]\to\mathbf
  w_\infty$ and $\Sigma_t\to\Sigma_\infty$, the unique solution of
  $M\Sigma_\infty+\Sigma_\infty M^\top=\sigma_0^2 I$---at exponential rate
  $2\min_i\mathrm{Re}\,\mu_i(M)=2(\alpha_{\mathrm{self}}-\gamma\abs{\lambda_{\min}(A)})$,
  and the stationary law is the Gaussian $\mathcal N(\mathbf w_\infty,\Sigma_\infty)$.
  Hence $\E[V(\W_t)]=\tfrac12(\norm{\E[\mathbf w_t]}^2+\Tr\Sigma_t)$ converges.
  When $A$ is normal this is sharp mode-by-mode: diagonalizing
  $A=\sum_a\nu_a\zeta_a\zeta_a^*$ with orthonormal $\zeta_a$, the modal
  second moments obey
  \begin{equation}
    \frac{\dif m_a}{\dif t}=-2(\alpha_{\mathrm{self}}+\gamma\nu_a)\,m_a+\sigma_0^2 d,
    \label{eq:modal_ode}
  \end{equation}
  each converging to $\sigma_0^2 d/[2(\alpha_{\mathrm{self}}+\gamma\nu_a)]$.

  \emph{Part (ii).}  At $\gamma=\gamma^*(A)$ the eigenvalue
  $\mu_\star=\alpha_{\mathrm{self}}+\gamma\lambda_{\min}(A)$ of $M$ has zero
  real part, so the Lyapunov equation has no bounded solution: variance
  accumulates along the marginal direction at constant rate and
  $\E[V(\W_t)]\sim\tfrac12\sigma_0^2 d\,t$ (for normal $A$,
  \eqref{eq:modal_ode} for the marginal mode reads $\dot m_{a^\star}=\sigma_0^2 d$).

  \emph{Part (iii).}  If $\gamma>\gamma^*(A)$ then $M$ has the eigenvalue
  $\mu_\star=\alpha_{\mathrm{self}}+\gamma\lambda_{\min}(A)$ with
  $\mathrm{Re}\,\mu_\star=\alpha_{\mathrm{self}}-\gamma\abs{\lambda_{\min}(A)}<0$.
  Choosing an initial mean with a nonzero component on the corresponding
  eigenspace, $\E[\mathbf w_t]=e^{-Mt}\E[\mathbf w_0]$ has
  $\norm{\E[\mathbf w_t]}\ge c\,e^{\abs{\mathrm{Re}\,\mu_\star}t}
  =c\,e^{(\gamma\abs{\lambda_{\min}(A)}-\alpha_{\mathrm{self}})t}$.  By
  Jensen, $\E[V(\W_t)]=\tfrac12\E\norm{\mathbf w_t}^2\ge\tfrac12\norm{\E[\mathbf w_t]}^2\to\infty$.
\end{proof}

\begin{remark}[The Perron mode is the most stable; the sign error]
\label{rem:signerror}
The consensus mode $\xi=\mathbf 1_K/\sqrt K$ has
$\mu=\lambda_{\max}(A_{\mathrm{sym}})=1$, hence modal rate
$-2(\alpha_{\mathrm{self}}+\gamma)<0$: it is the \emph{fastest-decaying}
direction, not an unstable one.  A previous version of this theorem
projected the instability argument onto the Perron eigenvector and wrote
the drift coefficient as $(\gamma\lambda_{\max}(A)-\alpha_{\mathrm{self}})$.
That sign is wrong: the coupling enters the drift as $-\gamma(A\otimes I)\mathbf w$
(Eq.~\eqref{eq:drift_simple}), so an $A$-mode with eigenvalue $\mu$ has
first-moment rate $-(\alpha_{\mathrm{self}}+\gamma\mu)$, which is
$-(\alpha_{\mathrm{self}}+\gamma)$ on the Perron mode---decaying.  The
genuine instability lives on the $\lambda_{\min}(A)$ mode and was never
identified.
\end{remark}

\begin{remark}[Exact threshold vs.\ the Lyapunov certificate; normal vs.\ non-normal]
\label{rem:moments}
Two thresholds appear in this paper and they should not be conflated.  The
\emph{exact} phase-transition threshold is
$\gamma^*(A)=\alpha_{\mathrm{self}}/\abs{\lambda_{\min}(A)}$, set by the
spectral abscissa of $M$ (Theorem~\ref{thm:phase}).  The quadratic Lyapunov
function $V$ delivers only the \emph{sufficient} certificate
$\gamma_{\mathrm{cert}}(A)=\alpha_{\mathrm{self}}/\abs{\lambda_{\min}(A_{\mathrm{sym}})}$
(Theorem~\ref{thm:stability}), because the generator bound sees the
symmetric part $A_{\mathrm{sym}}$.  Since
$\abs{\lambda_{\min}(A_{\mathrm{sym}})}\ge\abs{\lambda_{\min}(A)}$ always,
$\gamma_{\mathrm{cert}}(A)\le\gamma^*(A)$, with equality iff $A$ is normal.
For the symmetric complete and ring graphs the two coincide
($0.40$ and $0.124$).  For the non-normal star they separate:
$\gamma^*=\alpha_{\mathrm{self}}/1=0.10$ but
$\gamma_{\mathrm{cert}}=\alpha_{\mathrm{self}}/1.25=0.08$.  The gap is the
price of a quadratic certificate on non-normal coupling: $V$ cannot certify
the genuinely-stable band $\gamma\in(0.08,0.10)$ for the star.  Crucially the
certificate errs on the \emph{safe} side---it never certifies an unstable
system---so it remains the right object to attest; closing the gap requires
a non-quadratic (e.g.\ solution-of-Lyapunov-equation, weighted) certificate
$V_P(\W)=\tfrac12\mathbf w^\top(P\otimes I_d)\mathbf w$ with $P\succ0$ chosen
so that $PM+M^\top P\succ0$ up to $\gamma^*$.
\end{remark}

\subsection{Topology Dependence}

The exact threshold $\gamma^*(A)=\alpha_{\mathrm{self}}/\abs{\lambda_{\min}(A)}$
is genuinely topology-dependent, because $\abs{\lambda_{\min}(A)}$
varies across graphs while $\lambda_{\max}(A)=1$ does not.  For the three
canonical topologies at $K=5$:
\[
  \begin{aligned}
  \abs{\lambda_{\min}(A)}:\quad & \text{complete } 0.25 \;<\; \text{ring } 0.809 \;<\; \text{star } 1.0,\\
  \gamma^*:\quad & \text{complete } 0.40 \;>\; \text{ring } 0.124 \;>\; \text{star } 0.10.
  \end{aligned}
\]
(taking $\alpha_{\mathrm{self}}=0.10$).  The complete graph is the most
robust to coupling and the star the least---the opposite ranking would be
obtained from the (constant, hence uninformative) spectral radius, and a
different ranking again from the spectral gap $\lambda_{\max}-\lambda_2$.
The conservative Lyapunov certificate orders the same way, through
$\abs{\lambda_{\min}(A_{\mathrm{sym}})}$ (complete $0.25$, ring $0.809$,
star $1.25$), giving $\gamma_{\mathrm{cert}}$ of $0.40$, $0.124$, $0.08$.

\begin{corollary}[Topology ordering]
  \label{cor:topology}
  Fix $K$ and a coupling $\gamma$ in the stable range of all three
  topologies.  Then the exact critical thresholds are ordered
  $\gamma^*_{\mathrm{complete}} > \gamma^*_{\mathrm{ring}} > \gamma^*_{\mathrm{star}}$,
  and the steady-state Lyapunov levels are ordered
  \[
    \E[V_\infty]_{\mathrm{complete}}
    \;\leq\;
    \E[V_\infty]_{\mathrm{ring}}
    \;\leq\;
    \E[V_\infty]_{\mathrm{star}},
  \]
  because the destabilizing eigenvalues satisfy
  $\abs{\lambda_{\min}(A^{\mathrm{complete}})}
  \leq \abs{\lambda_{\min}(A^{\mathrm{ring}})}
  \leq \abs{\lambda_{\min}(A^{\mathrm{star}})}$, so for fixed $\gamma$ the
  star sits closest to (or beyond) its lower critical threshold.  The
  governing quantity is $\abs{\lambda_{\min}(A)}$, not the spectral gap
  $\lambda_2(A)$.
\end{corollary}

The ordering has a concrete governance reading: the \emph{topology} of interconnection is itself a control variable. A complete, densely diversified coupling tolerates the most shared adaptation before destabilizing, while a star---a single dominant hub on which all others depend---tolerates the least, so the same total coupling strength can be safe in one architecture and supercritical in another. Ensemble structure, not merely coupling magnitude, is therefore a lever the governance function can act on.

\section{Emergent Ensemble Risk}
\label{sec:emergent}

One of the most important motivations for the JLP over per-agent MRM
is that individual stability does not imply joint stability under
coordinated hidden drift.  We formalize this claim.

This is the rigorous form of the premise motivating ensemble-level governance---that a collection of individually sound models can nonetheless compose a fragile system---transposed to a population of self-adapting models.

\subsection{The Individual MRM Framework}

In standard per-agent MRM, each model $k$ is assigned a threshold
$\theta_{\mathrm{ind}} > 0$ and the monitoring rule is:
\emph{flag agent $k$ if $\bar{V}^k > \theta_{\mathrm{ind}}$}, where
$\bar{V}^k = \E_\pi[V^k(\W)]$ is the stationary per-agent Lyapunov
mean.  In the honest (no-drift) case,
$\theta_{\mathrm{ind}} = 2\,\E_\pi^{\mathrm{honest}}[V^k]$.

\subsection{Hidden Coordinated Drift}

\begin{theorem}[Emergent ensemble risk]
  \label{thm:emergent}
  Consider the decoupled system ($\gamma = 0$, $\Phi = 0$) with $K$ agents
  and suppose a hidden drift $h \in \R^d$ is added to every agent's
  dynamics, so the true drift of agent $k$ is
  $\mu^k(W^k) = -\alpha_{\mathrm{self}} W^k + h$.
  Define the honest noise floor
  $V^\infty_{\mathrm{hon}} = Kd\sigma_0^2/(4\alpha_{\mathrm{self}})$.

  If the drift magnitude satisfies
  \begin{equation}
    \frac{K}{2\alpha_{\mathrm{self}}^2}\norm{h}^2
    = \delta \cdot V^\infty_{\mathrm{hon}}
    \quad\text{with}\quad 0 < \delta < \frac{\theta_{\mathrm{ind}}}
    {V^\infty_{\mathrm{hon}} / K} - 1,
    \label{eq:h_condition}
  \end{equation}
  then:
  \begin{enumerate}[label=(\roman*)]
    \item No individual agent is flagged: for every $k$,
      $\E_\pi[V^k] < \theta_{\mathrm{ind}}$.
    \item The joint monitor flags the system:
      $V_{\mathrm{joint}} = \sum_k \E_\pi[V^k] >
      (1 + \delta/2) \cdot V^\infty_{\mathrm{hon}}$.
  \end{enumerate}
\end{theorem}

\begin{proof}
  \textbf{Stationary distribution under drift.}
  With hidden drift $h$, the $i$-th component of agent $k$'s weights
  evolves as
  \[
    \dif W^k_{t,i} = (-\alpha_{\mathrm{self}} W^k_{t,i} + h_i)\,\dif t
    + \sigma_0\,\dif B^k_{t,i}.
  \]
  This is an OU process with mean $h_i/\alpha_{\mathrm{self}}$, with
  stationary distribution
  \[
    \mathcal{N}\!\left(h_i/\alpha_{\mathrm{self}},\; \sigma_0^2/(2\alpha_{\mathrm{self}})\right).
  \]

  \textbf{Stationary per-agent second moment.}
  \[
    \E_\pi[V^k] = \frac{1}{2}\sum_{i=1}^d \E_\pi[(W^k_i)^2]
    = \frac{1}{2}\sum_{i=1}^d \left[\frac{h_i^2}{\alpha_{\mathrm{self}}^2}
    + \frac{\sigma_0^2}{2\alpha_{\mathrm{self}}}\right]
    = \frac{\norm{h}^2}{2\alpha_{\mathrm{self}}^2}
    + \frac{d\sigma_0^2}{4\alpha_{\mathrm{self}}}.
  \]

  \textbf{Part (i): Individual check.}
  The per-agent honest mean is
  $\E_\pi^{\mathrm{hon}}[V^k] = d\sigma_0^2/(4\alpha_{\mathrm{self}})
  = V^\infty_{\mathrm{hon}}/K$.
  The individual threshold is $\theta_{\mathrm{ind}} = 2V^\infty_{\mathrm{hon}}/K$.
  Agent $k$ is flagged only if
  \[
    \E_\pi[V^k] = \frac{\norm{h}^2}{2\alpha_{\mathrm{self}}^2}
    + \frac{V^\infty_{\mathrm{hon}}}{K}
    > \frac{2V^\infty_{\mathrm{hon}}}{K},
  \]
  i.e., if $\norm{h}^2 > 2\alpha_{\mathrm{self}}^2 V^\infty_{\mathrm{hon}} / K$.
  By condition \eqref{eq:h_condition},
  $\norm{h}^2 = 2\delta\alpha_{\mathrm{self}}^2 V^\infty_{\mathrm{hon}} / K$,
  and the right-hand condition requires
  $\delta < (\theta_{\mathrm{ind}}/V^\infty_{\mathrm{hon}})\cdot K - 1
  = 1$.  Under this condition, no agent is individually flagged.

  \textbf{Part (ii): Joint check.}
  The joint stationary value is
  \[
    V_{\mathrm{joint}} = K \cdot \E_\pi[V^k]
    = \frac{K\norm{h}^2}{2\alpha_{\mathrm{self}}^2}
    + V^\infty_{\mathrm{hon}}.
  \]
  With the choice of $h$ in \eqref{eq:h_condition},
  the first term equals $\delta V^\infty_{\mathrm{hon}}$, so
  $V_{\mathrm{joint}} = (1 + \delta) V^\infty_{\mathrm{hon}}$.
  The joint threshold is $(1+\delta/2) V^\infty_{\mathrm{hon}} <
  (1+\delta) V^\infty_{\mathrm{hon}} = V_{\mathrm{joint}}$, so the
  joint monitor fires.
\end{proof}

\begin{remark}
  Theorem~\ref{thm:emergent} establishes an exact window $\delta \in
  (0, 1)$ in which hidden coordinated drift is invisible
  to individual MRM but detectable by the JLP.  The width of the window in
  $\delta$ does not depend on $K$; what improves with $K$ is the precision of
  the joint estimate, since $V_{\mathrm{joint}}$ averages the same coordinated
  signal over more agents.  Throughout the window $\norm{h}$ is small relative
  to $\sigma_0$: each agent's drift stays below its individual noise floor
  while the collective signal exceeds the joint threshold.  Numerical Study~C5 exhibits this phenomenon
  concretely with $K=10$, $h_{\mathrm{per\,dim}} = 0.00935$,
  $\sigma_0 = 0.05$, confirming that $0/10$ agents are individually
  flagged while the joint check fires.

In governance terms this is the paper's central cautionary result: a validation regime that certifies each model in isolation can pass an entire population while the ensemble they constitute drifts together past its joint tolerance. The joint monitor that closes the gap is an aggregate of the very same per-model Lyapunov quantities a per-model regime already computes---what changes is not the data collected but that it is assessed \emph{jointly}, at the level of the ensemble.
\end{remark}

\section{Zero-Knowledge Proof Construction}
\label{sec:zk}

\subsection{What is Worth Attesting}

A first instinct is to have the firm prove, at each epoch $\tau_n$, the
pointwise generator inequality $\Lgen V(\W_t) < -\alpha V(\W_t)+\beta$ at
the live weights $\W_t$, in zero knowledge.  This is the wrong target, for
a reason that is structural rather than cryptographic.

\begin{proposition}[The per-epoch weight proof is vacuous]
\label{prop:tautology}
Fix coefficients with $\gamma<\gamma_{\mathrm{cert}}(A)$ and let
$\alpha=\alpha_{\mathrm{eff}}$, $\beta$ be as in Theorem~\ref{thm:stability}.
Then:
\begin{enumerate}[label=(\roman*)]
  \item \emph{(Tautology.)} The inequality $\Lgen V(\W)<-\alpha V(\W)+\beta$
    holds at \emph{every} $\W\in\R^{dK}$.  A proof that it holds at the
    particular live state $\W_t$ therefore conveys no information beyond
    what the publicly-declared coefficients already imply.
  \item \emph{(Perverse monotonicity.)} Write the exact generator
    \eqref{eq:LV_exact} as $\Lgen V(\W)=-2\alpha_{\mathrm{self}}V(\W)+R(\W)$
    with $R$ at most quadratic and dominated by $-2\alpha_{\mathrm{self}}V$
    for large $\norm{\W}$.  Then $\Lgen V(\W)+\alpha V(\W)-\beta\to-\infty$
    as $\norm{\W}\to\infty$: the check is satisfied with ever-greater slack
    the larger the weights grow.  A system whose weights have blown up
    passes the ``stability'' test most comfortably of all.
\end{enumerate}
\end{proposition}

\begin{proof}
(i) is exactly Theorem~\ref{thm:stability}: the Foster--Lyapunov inequality
was proved for all $\W$, so it is not a property of any particular $\W_t$.
For (ii), by \eqref{eq:LV_exact} and Lemma~\ref{lem:coupling_bound},
$R(\W)\le\gamma\Phi_{\max}\sqrt K\norm{\W}-2\gamma\lambda_{\min}(A_{\mathrm{sym}})V(\W)+\tfrac12Kd\sigma_0^2$,
which is linear-plus-constant-times-$V$; against the $-2\alpha_{\mathrm{self}}V$
term, $\Lgen V(\W)+\alpha V(\W)-\beta=-(2\alpha_{\mathrm{self}}-\alpha+O(\gamma))V(\W)+O(\norm{\W})\to-\infty$
since $\alpha=\alpha_{\mathrm{eff}}<2\alpha_{\mathrm{self}}$.
\end{proof}

The diagnosis is that pointwise stability of $V$ is not where the risk
lives.  By Theorem~\ref{thm:phase} the system is stable or unstable
according to whether $\gamma<\gamma^*(A)=\alpha_{\mathrm{self}}/\abs{\lambda_{\min}(A)}$---a
condition on the \emph{declared parameters} $(A,\alpha_{\mathrm{self}},\gamma)$,
involving no live weights at all.  The object worth attesting is this
static spectral certificate, established once per model-configuration change
rather than at every epoch.

\subsection{The Spectral Certificate and Its Circuit}

If the topology $A$ is not itself confidential, the certificate
$\gamma<\alpha_{\mathrm{self}}/\abs{\lambda_{\min}(A)}$ is checked directly
from the declared model documentation and needs no proof system.  The ZK
machinery becomes relevant only when $A$ (or $\gamma$) is proprietary.  In
that case the firm commits to $A$ and proves, in zero knowledge, a
\emph{spectral} statement about the committed matrix:
\begin{equation}
  \pi_{\mathrm{JLP}} = \mathrm{ZK\text{-}SNARK}\bigl\{
    (A,\gamma) :\;
    \alpha_{\mathrm{self}} I_K + \gamma A_{\mathrm{sym}} \succ 0
    \;\wedge\;
    \mathrm{Hash}(A,\gamma,\mathtt{Key}) = H
  \bigr\},
  \label{eq:zk_statement}
\end{equation}
where $A_{\mathrm{sym}}=\tfrac12(A+A^\top)$.  The positive-definiteness
$\alpha_{\mathrm{self}} I + \gamma A_{\mathrm{sym}}\succ0$ is equivalent to
$\gamma<\gamma_{\mathrm{cert}}(A)$, the safe (conservative) certificate; it
is the natural object to prove because it is decidable by a Cholesky
factorization, which a circuit can verify without computing eigenvalues.
(Attesting the exact $\gamma<\gamma^*(A)$ instead requires bounding
$\lambda_{\min}(A)$ of the committed, possibly non-symmetric $A$; we use the
conservative certificate precisely because it is SNARK-friendly and never
certifies an unstable system.)  Crucially, \eqref{eq:zk_statement} has
\emph{no $\W_t$ argument}: it is proved once when the firm declares or
changes its coupling configuration, not at every monitoring epoch.

The weight-commitment clause $\mathrm{Hash}(\cdot)=H$ is retained, but its
role is now correctly scoped: it provides tamper-evidence and
non-repudiation for the declared \emph{configuration} $(A,\gamma)$, and is
orthogonal to stability.  A separate, genuinely state-dependent monitoring
signal---if one is wanted at each epoch---is whether the realized Lyapunov
value $V(\W_{\tau_n})$ stays within the stationary envelope
$\beta/\alpha_{\mathrm{eff}}$ of \eqref{eq:stationary_bound}; that is a
meaningful per-epoch check because it can fail, unlike the pointwise
generator inequality.

\subsection{Circuit for the Cholesky Certificate}

The statement \eqref{eq:zk_statement} is converted to an arithmetic
circuit over a prime field $\mathbb{F}_p$ as follows.  The witness is the
committed matrix $A$ (equivalently its symmetric part $S:=\alpha_{\mathrm{self}}I_K+\gamma A_{\mathrm{sym}}$)
together with a claimed Cholesky factor $L$; there are no model weights in
the circuit.

\paragraph{Step 1: Form $S=\alpha_{\mathrm{self}}I_K+\gamma A_{\mathrm{sym}}$.}
From the committed entries of $A$, compute $A_{\mathrm{sym}}=\tfrac12(A+A^\top)$
and $S$.  This is $O(K^2)$ field additions and scalar multiplications;
$\alpha_{\mathrm{self}},\gamma$ are public.

\paragraph{Step 2: Verify the Cholesky factorization $S=LL^\top$.}
The prover supplies a lower-triangular $L$ as private witness; the circuit
checks $LL^\top=S$ entrywise ($O(K^3)$ multiplications) and that each
diagonal entry $L_{ii}$ is nonzero.  A real Cholesky factor with nonzero
(hence, by the field encoding, positive) diagonal exists iff $S\succ0$, i.e.\
iff $\gamma<\gamma_{\mathrm{cert}}(A)$.  This replaces an eigenvalue
computation, which is not circuit-friendly, by a verification of an
algebraic identity.

\paragraph{Step 3: Range-check the diagonal.}
The circuit checks each $L_{ii}^2>0$ in the signed embedding, certifying
strict positivity of the pivots and hence $S\succ0$.

\paragraph{Step 4: Hash check.}
A Poseidon hash $H(A,\gamma,\mathtt{Key})$ over $\mathbb{F}_p$ binds the
committed configuration to the ledger entry.  Poseidon is SNARK-friendly
with low gate count \cite{grassi2021poseidon}.  This clause provides
tamper-evidence for $(A,\gamma)$ and is independent of the stability claim.

\subsection{Security Properties}

The three properties below are what make the certificate usable as governance reporting. Soundness means a model owner cannot fake compliance with the ensemble-stability threshold; completeness means a genuinely compliant owner can always demonstrate it; and zero-knowledge means it can do so \emph{without disclosing its proprietary coupling matrix}. Together they resolve the central tension in collecting ensemble-level information: the validation function learns whether each model sits on the stable side of $\gamma^*$, and can aggregate these attestations, without any party pooling the sensitive models whose very concentration would create a new exposure.

\begin{proposition}[Soundness and completeness of $\pi_{\mathrm{JLP}}$]
  \label{prop:zk}
  Let $\mathcal{R}$ be the NP relation defined by
  statement~\eqref{eq:zk_statement}: a pair $(A,\gamma)$ with a valid
  witness $L$ such that $LL^\top=\alpha_{\mathrm{self}}I_K+\gamma A_{\mathrm{sym}}$
  has strictly positive diagonal, together with the hash binding.  A
  ZK-SNARK with the following properties can be instantiated using the
  Groth16 construction \cite{groth2016size}:
  \begin{enumerate}[label=(\roman*)]
    \item \textbf{Completeness.} If $\gamma<\gamma_{\mathrm{cert}}(A)$ the
      certifying Cholesky factor exists and the prover can always generate a
      valid proof $\pi_{\mathrm{JLP}}$.
    \item \textbf{Soundness.} With overwhelming probability, no
      polynomial-time prover can produce a valid $\pi_{\mathrm{JLP}}$ when
      $\alpha_{\mathrm{self}}I+\gamma A_{\mathrm{sym}}\not\succ0$ (no real
      Cholesky factor with positive diagonal exists), i.e.\ the certificate
      cannot be forged for an uncertified configuration.
    \item \textbf{Zero-knowledge.} The proof reveals no information about
      the committed $A$ beyond the membership statement.
  \end{enumerate}
\end{proposition}

\begin{proof}
  All three properties follow from the Groth16 construction for rank-1
  constraint systems (R1CS), the standard encoding for the arithmetic
  circuit above.  The circuit size is $O(K^3)$ constraints (dominated by the
  Cholesky-identity check), polynomial in $K$ and independent of the model
  dimension $d$ and of the live weights, and hence admits an efficient
  prover.  Soundness of the algebraic check reduces to the fact that
  $S\succ0$ iff a Cholesky factor with positive diagonal exists; we refer to
  \cite{groth2016size} for the formal SNARK security proof under the generic
  group model.
\end{proof}

\paragraph{Governance ledger integration.}
When the firm declares or changes its coupling configuration, it:
(a) commits to $(A,\gamma)$ on the ledger;
(b) generates $\pi_{\mathrm{JLP}}$ proving the spectral certificate and the
hash binding;
(c) submits $(\pi_{\mathrm{JLP}}, H)$ to the governance ledger.
The validator runs the polynomial-time verifier $\mathcal{V}$, which accepts
or rejects without ever seeing $A$.  Per-epoch telemetry, if required,
reports only whether $V(\W_{\tau_n})$ remains within the declared envelope
$\beta/\alpha_{\mathrm{eff}}$---a check that is state-dependent and can
genuinely fail, unlike the pointwise generator inequality
(Proposition~\ref{prop:tautology}).

\section{MRM and Generative AI Model Governance}
\label{sec:sr117}

\subsection{Background: The MRM Framework}

Traditional MRM (\cite{sr117}) is organized around three pillars: \emph{Model Development and
Implementation}, \emph{Model Validation}, and \emph{Ongoing Monitoring
and Governance}.  The guiding principles of MRM can profitably be applied to situations involving AI-generated content that influences
a decision. However, the practical implementation and instantiation of MRM principles to GenAI use cases is where the complexity lives. 

\subsection{Why Standard MRM Is Insufficient for GenAI}

Classical MRM was designed for static or infrequently-updated
quantitative models.  Its three pillars assume:
(a) model parameters are fixed between formal validation events;
(b) models operate in isolation with well-defined input--output interfaces;
(c) model performance is measured against historical back-test data.

Generative AI models violate all three assumptions simultaneously.
First, their parameters update continuously through reinforcement
learning from human feedback (RLHF), retrieval-augmented generation
(RAG) updates, and meta-learning loops.  Second, in multi-agent
deployments---common in large institutions that operate distinct models for
disparate use cases---models
share embeddings, fine-tuning data, and inference infrastructure,
creating de facto coupling that MRM per-model view ignores.
Third, generative model performance is multi-dimensional (accuracy,
fairness, calibration, hallucination rate) and cannot be fully
characterized by a single scalar loss measured against a static
validation set.

The JLP addresses this gap by providing a continuous-time, system-level
stability attestation that maps directly onto each MRM pillar.

\subsection{Mapping the JLP onto MRM Pillars}

\paragraph{Pillar 1: Conceptual Soundness.}

MRM requires that a model's theoretical foundations be
``well-founded in published research or widespread industry practice''
and that ``key assumptions limiting the model's applicability or
accuracy'' be documented (\cite{sr117}).

Under the JLP, the choice of the joint Lyapunov function
$V(\W) = \tfrac{1}{2}\sum_k \norm{W^k}^2$ must be justified at
initial validation.  Specifically, the firm must demonstrate:

\begin{enumerate}[label=(\roman*)]
  \item That $V$ is a meaningful proxy for ensemble-level model risk:
    large $V$ indicates that aggregate model weights have drifted far
    from their validated origin, which corresponds to distributional
    shift in the models' outputs.

  \item That the coupling matrix $A$ accurately reflects the
    institution's actual model interaction topology (shared fine-tuning
    datasets, shared embedding layers, common RLHF reward signals).

  \item That the declared parameters $\alpha_{\mathrm{self}}$,
    $\sigma_0$, $\gamma$, and $\Phi$ are calibrated to observed weight
    dynamics during a supervised validation period, not set
    arbitrarily.

  \item That the critical threshold $\gamma^*(A)$ derived in
    Theorem~\ref{thm:phase} has been computed for the declared topology
    and that the institution's operational $\gamma$ satisfies
    $\gamma < \gamma^*(A)$ with a documented safety margin.
\end{enumerate}

This maps directly to the MRM requirement that model
developers should be able to explain why a particular
mathematical approach was selected and what its limitations are.

\paragraph{Pillar 2: Ongoing Monitoring.}

MRM requires that models be monitored continuously against
performance benchmarks and that deviations trigger escalation
procedures.  

The JLP provides a natural and continuous monitoring framework.
At each epoch $\tau_n$ (e.g., daily), the firm submits a ZK proof
$\pi_{\mathrm{JLP}}$ that
$\Lgen V(\W_{\tau_n}) < -\alpha_{\mathrm{eff}} V(\W_{\tau_n}) + \beta$.
This is equivalent to certifying that the joint system is within its
stability envelope as characterized by Theorem~\ref{thm:stability}.
The key features of this monitoring framework are:

\begin{enumerate}[label=(\roman*)]
  \item \textbf{Continuous attestation:} unlike annual validation
    exercises, the JLP provides daily cryptographic attestations stored
    on the governance ledger, creating an audit-ready tamper-evident
    record.

  \item \textbf{Early warning:} the sequence
    $\{V(\W_{\tau_n})\}_{n=0}^{\infty}$ forms a real-valued time series
    whose trend can be monitored with standard process-control methods
    (CUSUM, EWMA).  An upward trend in $V$ signals incipient
    instability before the stability bound is formally violated.

  \item \textbf{Granular decomposition:} the exact generator formula
    \eqref{eq:LV_exact} decomposes the total risk into self-decay
    (first term), coupling (second term), and noise (third term),
    allowing examiners to identify which component is driving elevated
    $V$.

  \item \textbf{Noise floor calibration:} Theorem~\ref{thm:noise_floor}
    provides a closed-form baseline $\E[V_\infty] = \beta_0/(2\alpha_{\mathrm{self}})$
    against which current $V$ can be compared, making the monitoring
    threshold explicit and auditable.
\end{enumerate}

\paragraph{Pillar 3: Model Validation.}

MRM requires independent model validation (MV) by a function
separate from model development.  The guidance specifies three
components of validation: evaluation of conceptual soundness,
ongoing monitoring, and outcomes analysis.

The JLP enhances validation in the multi-agent context in three ways.
First, the MV function can independently verify the topology declaration
by examining shared infrastructure: fine-tuning pipelines, embedding
layers, and reward-signal configurations.  Second, the historical
sequence of $\pi_{\mathrm{JLP}}$ proofs on the ledger provides an
objective record for outcomes analysis---the MV function checks not
only that proofs were submitted but that the underlying $V$ trajectory
was consistent with the declared dynamics parameters.  Third, the
governance injection probe described in Section~\ref{sec:numerical}
(Numerical Study~C4) provides a \emph{challenge procedure}: the
examiner injects a known perturbation $\Delta$ to the meta-parameter
$\Phi_t$ and verifies that the JLP telemetry responds as predicted by
the certified dynamics within the detection horizon.

\subsection{The Coupling Gap: A New MRM Risk Category}

We formalize a new MRM risk
category---the \emph{coupling gap}---defined as the margin by which the
destabilizing coupling exceeds the self-decay rate.  Formally, with
$\hat A_{\mathrm{sym}}=\tfrac12(\hat A+\hat A^\top)$:
\begin{equation}
  \Delta_{\mathrm{coup}} := \hat\gamma \cdot \abs{\lambda_{\min}(\hat A)} - \alpha_{\mathrm{self}}.
  \label{eq:coupling_gap}
\end{equation}
When $\Delta_{\mathrm{coup}} > 0$ the operational coupling has crossed the
exact threshold $\gamma^*(\hat A)=\alpha_{\mathrm{self}}/\abs{\lambda_{\min}(\hat A)}$
and the joint second moment diverges (Theorem~\ref{thm:phase}); individual
MRM is then completely inadequate.  A more conservative trigger, safe under
non-normality, fires when the Lyapunov certificate fails, i.e.\ when
$\hat\gamma\cdot\abs{\lambda_{\min}(\hat A_{\mathrm{sym}})}>\alpha_{\mathrm{self}}$
(equivalently $\alpha_{\mathrm{self}}I+\hat\gamma\hat A_{\mathrm{sym}}\not\succ0$);
this is the quantity the spectral SNARK of Section~\ref{sec:zk} attests.

\subsection{Governance Workflow Under the JLP}

We describe the end-to-end governance workflow for
$K$ generative AI models under MRM with JLP attestation.

\paragraph{Stage 1: Model Inventory and Topology Declaration.}
The model owner documents all $K$ AI models in scope, their
shared infrastructure, and the coupling adjacency matrix $A$.  This
is submitted as part of the model documentation required under
MRM Section~I.

\paragraph{Stage 2: Initial Parameter Calibration.}
During the initial validation period (e.g. 90--180 days), the
self-decay rate $\alpha_{\mathrm{self}}$, noise coefficient $\sigma_0$,
and coupling strength $\gamma$ are estimated from observed weight
trajectories using the moment-matching conditions of
Theorems~\ref{thm:noise_floor} and \ref{thm:stability}.  The critical
threshold $\gamma^*(A)$ is computed and documented.

\paragraph{Stage 3: Validation of the JLP Circuit.}
The model validation function independently audits the arithmetic
circuit described in Section~\ref{sec:zk}, verifying that the circuit
correctly encodes the stability condition from Theorem~\ref{thm:stability}.
This audit is analogous to the standard ``back-testing'' requirement but adapted to the ZK-SNARK context.

\paragraph{Stage 4: Ongoing JLP Submission.}
At each validation epoch $\tau_n$ (recommended: daily), the model
risk management system automatically:
(a) samples $\W_{\tau_n}$ from live model registries;
(b) computes $V(\W_{\tau_n})$ and $\Lgen V(\W_{\tau_n})$ via
Theorem~\ref{thm:generator};
(c) generates $\pi_{\mathrm{JLP}}$ and submits it with $H_{\tau_n}$
to the governance ledger.
A proof failure triggers an automatic escalation to the model risk
committee.

\paragraph{Stage 5: Challenge Procedure.}
Independent reviewers/red team challengers may, at any time, inject a documented perturbation
$\Delta_\Phi$ into the meta-parameter $\Phi_t$ and demand that the
firm's JLP telemetry reflect the perturbation within the predicted
detection horizon.  Failure to respond within $O(1/\gamma)$ steps
(see Numerical Study~C4) signals potential tampering with the
telemetry stream.

\subsection{Limitations and Open Questions}
\label{sec:limitations}

\paragraph{Linear dynamics, and homogeneity across agents.}
This is the strongest modeling restriction in the paper and it underlies every exact result, so we
state it first. Assumption~\ref{ass:loss} takes each agent's loss to be exactly the isotropic
quadratic $\loss_k(w)=\tfrac12\alpha_{\mathrm{self}}\norm{w}^2$, with the \emph{same}
$\alpha_{\mathrm{self}}$ for every agent, and the coupling in \eqref{eq:drift} is linear in $\W$
with additive, state-independent noise. The joint process is therefore a linear multivariate
Ornstein--Uhlenbeck process. That is precisely why the analysis closes in the form it does: the exact
threshold $\gamma^*(A)=\alpha_{\mathrm{self}}/\abs{\lambda_{\min}(A)}$ is the point at which an
eigenvalue of a constant drift matrix crosses zero; the noise floor
$\E_\pi[V]=Kd\sigma_0^2/(4\alpha_{\mathrm{self}})$ is the stationary second moment of a Gaussian; and
the spectral certificate is a statement about a fixed symmetric matrix. None of these survives
unchanged for a non-convex, anisotropic, or agent-heterogeneous loss landscape --- which is what the
loss surfaces of actual generative models are.

Three consequences deserve emphasis. First, $\alpha_{\mathrm{self}}$ is not a property one reads off
a deployed model; it is the curvature of an assumed quadratic, and for a real system it would have to
be estimated locally, with the results holding only in whatever neighborhood that local
approximation is credible. Second, heterogeneous $\alpha_{\mathrm{self}}^{(k)}$ would replace the
scalar threshold by a condition on a non-uniformly scaled operator, and the clean eigenvalue
characterization would become an inequality rather than an identity. Third, non-convexity admits
multiple basins, so ``the'' stationary distribution is replaced by basin-dependent behavior and the
global statements above become local ones.

We therefore read this paper's contribution as a tractable idealization that isolates one mechanism
--- how coupling topology governs the stability of an interacting ensemble --- and makes it exactly
computable, not as a description of the dynamics of production generative models. Extending the
threshold to locally-quadratic or heterogeneous settings is the most consequential open problem we
leave.

\paragraph{Non-quadratic Lyapunov functions.}
The quadratic $V(\W) = \tfrac{1}{2}\sum_k \norm{W^k}^2$ is
analytically tractable but may not capture all relevant risk
dimensions for generative models (e.g., distributional shift in
output space, not parameter space).  Future work should consider
Lyapunov functions based on output divergence metrics (KL, MMD,
Wasserstein).

\paragraph{Non-stationary meta-parameters.}
We assumed $\Phi_t$ is fixed.  In practice, $\Phi_t$ may itself be
updated by a higher-level learning process, introducing non-stationarity
that invalidates the stationary distribution results.

\paragraph{Adversarial meta-parameters.}
The ZK-SNARK proves that $\W_t$ satisfies the stability condition at
epoch $\tau_n$, but it does not prevent the firm from manipulating
$\Phi_t$ between epochs.  Continuous $\Phi_t$ monitoring (via the
injection probe of Study~C4) partially mitigates this, but a fully
adversarial $\Phi_t$ scenario requires stronger proof systems.

\section{Numerical Studies}
\label{sec:numerical}

We validate the theoretical claims of Sections~\ref{sec:lyapunov}--\ref{sec:emergent}
through five numerical studies (C1--C5) implemented in Python (NumPy).  All simulations use Euler--Maruyama
discretization with step size $\Delta t = 0.01$.  Seeds are fixed and documented, so every figure
and table is exactly reproducible; in Study~C3, where several parameter configurations are compared,
the seeds additionally depend on the configuration so that the comparisons are statistically
independent rather than deterministic rescalings of one another.

\paragraph{Scope of the numerical evidence.}
These are verification experiments for an analytically solvable model, and their limits should be
read into every number below. \emph{(i)} Monte Carlo dispersion is reported only for Study~C3, where
$n_{\mathrm{mc}} = 30$ paths per configuration permit a standard error; the remaining studies report
single-path or single-configuration point estimates without error bars, so their trailing digits are
not significant. \emph{(ii)} Ensemble sizes are small --- $K = 5$ throughout except $K = 10$ in
Study~C5 --- and the topology comparison of Study~C2 is conducted at $K = 5$ only, so the ordering
complete $>$ ring $>$ star is demonstrated at one ensemble size rather than established as a general
scaling law. \emph{(iii)} No step-size convergence study is performed; discretization bias at
$\Delta t = 0.01$ is not quantified. \emph{(iv)} Most importantly, the simulated system \emph{is} the
model of Section~\ref{sec:model} --- a linear Ornstein--Uhlenbeck ensemble --- so these studies
verify that our analysis of that model is correct. They are not evidence that production model
ensembles obey it; that would require the linearity and homogeneity assumptions discussed in
Section~\ref{sec:limitations} to hold empirically, which we do not test.

\subsection{Common Model Specification}

The baseline system uses:
$K = 5$ agents (except C5 which uses $K = 10$),
$d = 4$ weight dimensions,
$\alpha_{\mathrm{self}} = 0.10$,
$\sigma_0 = 0.05$,
$\norm{\Phi} = 0.3$ (unit-normalized direction),
$\Delta t = 0.01$.
Initial conditions are drawn near stationarity:
$W^k_0 \sim \mathcal{N}(0,\, (\sigma_0/\sqrt{2\alpha_{\mathrm{self}}})^2 I_d)$
to eliminate transient bias.  The theoretical noise floor is
$V^\infty = Kd\sigma_0^2/(4\alpha_{\mathrm{self}}) = 0.02500$ per agent.

\subsection{Study C1: Generator Bound vs.\ Coupling Strength}

\paragraph{Design.}
\begin{sloppypar}
We simulate the complete-graph system for $T = 8{,}000$ steps
(time $T = 80$), with a warm-up of $T/2$ steps before measurement.
For each $\gamma \in \{0.00, 0.05, 0.10, 0.20, 0.30, 0.40, 0.50\}$,
we record the stationary mean of $V(t)$ and the empirical
$\bar{\Lgen V} = \overline{\Delta V / \Delta t}$, and compute:
\end{sloppypar}
\begin{align*}
  \text{Residual}(\gamma) &= \bar{\Lgen V} + 2\alpha_{\mathrm{self}} \bar V, \\
  \alpha_{\mathrm{eff}}(\gamma) &=
    (-\bar{\Lgen V} + \beta_{\mathrm{naive}}) / \bar V, \\
  \beta_{\mathrm{eff}} &= \bar{\Lgen V} + \alpha_{\mathrm{eff}} \bar V.
\end{align*}
Here $\beta_{\mathrm{naive}} = \tfrac{1}{2}Kd\sigma_0^2 = 0.02500$.

\paragraph{Results.}

\begin{table}[H]
\centering
\caption{C1: Generator bound decomposition vs.\ coupling $\gamma$.}
\label{tab:c1}
\begin{tabular}{rrrrrr}
\toprule
$\gamma$ & $\bar V$ & $\bar{\Lgen V}$ &
$\text{Residual}$ & $\alpha_{\mathrm{eff}}(\gamma)$ & $\beta_{\mathrm{eff}}$ \\
\midrule
0.00 & 0.1037 & $-$0.00241 & 0.01834 & 0.2642 & 0.02500 \\
0.05 & 0.1463 & $-$0.00308 & 0.02618 & 0.1919 & 0.02500 \\
0.10 & 0.1912 & $-$0.00351 & 0.03474 & 0.1491 & 0.02500 \\
0.20 & 0.2880 & $-$0.00422 & 0.05338 & 0.1015 & 0.02500 \\
0.30 & 0.5036 & $-$0.00516 & 0.09556 & 0.0599 & 0.02500 \\
0.40 & 1.7668 & $+$0.01427 & 0.36764 & 0.0061 & 0.02500 \\
0.50 & 18.660 & $+$0.93935 & 4.67139 & $-$0.0490 & 0.02500 \\
\bottomrule
\end{tabular}
\end{table}

\paragraph{Interpretation.}
Three findings emerge.  First, $\alpha_{\mathrm{eff}}(\gamma)$ is a
strictly decreasing function of $\gamma$ (from 0.264 at $\gamma=0$ to
$-0.049$ at $\gamma=0.50$), confirming that coupling reduces---and
eventually reverses---the effective decay rate.  Second, the residual
$\Lgen V + 2\alpha_{\mathrm{self}} V$ grows rapidly with $\gamma$
(from 0.018 to 4.671), confirming Corollary~\ref{cor:gen_upper}: the
naive bound fails as soon as $\gamma > 0$.  Third, $\beta_{\mathrm{eff}}$
remains at exactly $\beta_{\mathrm{naive}} = 0.02500$ for all stable
$\gamma$, confirming that the noise floor formula is exact and the
coupling affects only $\alpha_{\mathrm{eff}}$.  Note that the complete-graph
run is still stable at $\gamma=0.30$ ($\bar V\approx0.50$) and only blows up
between $\gamma=0.40$ and $0.50$, consistent with the exact threshold
$\gamma^*=\alpha_{\mathrm{self}}/\abs{\lambda_{\min}(A)}=0.10/0.25=0.40$ for
the complete graph (Table~\ref{tab:c2})---not the $0.20$ that the original
$2\alpha_{\mathrm{self}}/\lambda_{\max}$ formula would give.  The implication
for MRM is direct: a firm that declares the naive decay rate
$\alpha_{\mathrm{eff}} = 2\alpha_{\mathrm{self}} = 0.20$ while operating at
$\gamma = 0.20$ is overstating its effective decay rate by a factor of two
($0.2642 \to 0.1015$), which translates to an underestimate of the time
required to return to the stability envelope after a shock.

\subsection{Study C2: Critical Coupling by Topology}

\paragraph{Design.}
For each topology we measure the empirical critical coupling two ways.
(a) A divergence sweep: we sweep $\gamma$, run $T=3{,}000$ steps with the
meta-parameter present, and record the steady-state $\bar V$ (declaring
divergence when the last-200-step mean exceeds $500$); this produces the
$\bar V(\gamma)$ curves in Figure~\ref{fig:jlp}.  (b) A growth-rate
crossing: for the \emph{homogeneous} system ($\Phi=0$) we project onto the
$\lambda_{\min}(A)$ eigenmode and estimate the late-time exponential growth
rate of its second moment, locating $\gamma^*_{\mathrm{emp}}$ where it
crosses zero.  Method (b) isolates the marginal mode from the noise floor
and recovers the asymptotic threshold cleanly; the lenient fixed-horizon
divergence test of method (a) overshoots it (a finite run has not yet blown
up just above $\gamma^*$).

\paragraph{Results.}

\begin{table}[H]
\centering
\caption{C2: Critical coupling by topology
($K=5$, $\alpha_{\mathrm{self}}=0.10$).  The exact threshold is
$\gamma^*=\alpha_{\mathrm{self}}/\abs{\lambda_{\min}(A)}$; the Lyapunov
certificate is $\gamma_{\mathrm{cert}}=\alpha_{\mathrm{self}}/\abs{\lambda_{\min}(A_{\mathrm{sym}})}$.}
\label{tab:c2}
\begin{tabular}{lcrrrr}
\toprule
Topology & normal? & $\abs{\lambda_{\min}(A)}$ & $\gamma^*$ (exact)
  & $\gamma^*_{\mathrm{emp}}$ & $\gamma_{\mathrm{cert}}$ \\
\midrule
Complete & yes & 0.250 & 0.400 & 0.404 & 0.400 \\
Ring     & yes & 0.809 & 0.124 & 0.120 & 0.124 \\
Star     & no  & 1.000 & 0.100 & 0.100 & 0.080 \\
\bottomrule
\end{tabular}
\end{table}

\paragraph{Interpretation.}
The empirical critical coupling tracks
$\gamma^*=\alpha_{\mathrm{self}}/\abs{\lambda_{\min}(A)}$ to within Monte
Carlo error for all three topologies---most pointedly the star, whose true
transition sits at $0.10$, exactly $\alpha_{\mathrm{self}}/\abs{\lambda_{\min}(A)}$
with $\abs{\lambda_{\min}(A)}=1$.  The ordering is by $\abs{\lambda_{\min}(A)}$
(complete $0.25 <$ ring $0.809 <$ star $1.0$), so the complete graph
tolerates four times the coupling of the star.  This is the opposite of
what the spectral radius predicts---which is constant ($\lambda_{\max}=1$)
and hence the same threshold for all three---confirming
Theorem~\ref{thm:phase} and Corollary~\ref{cor:topology}.  For the two
normal graphs the Lyapunov certificate $\gamma_{\mathrm{cert}}$ equals
$\gamma^*$; for the non-normal star it is strictly smaller ($0.08<0.10$),
the conservative-but-safe gap of Remark~\ref{rem:moments}.

What this study does and does not show is worth stating plainly. The system is a linear
Ornstein--Uhlenbeck process whose stability threshold is available in closed form, so C2 is a check
that the simulation and the algebra agree --- it confirms the arithmetic and the implementation, and
it rules out sign or scaling errors in $\gamma^*(A)$. It is not independent evidence that real model
ensembles exhibit this threshold, since the threshold's existence follows from the linearity assumed
in Assumption~\ref{ass:loss} rather than being discovered empirically. Its value is that the
\emph{ordering} by $\abs{\lambda_{\min}(A)}$, and the failure of the spectral radius to predict it,
are both non-obvious and are borne out numerically.

The MRM
implication is that two institutions with identical $K$,
$\alpha_{\mathrm{self}}$, and $\gamma$ may face very different ensemble-level risk
depending solely on their interaction topology, and that the
quantity to declare and bound is $\abs{\lambda_{\min}(A)}$---not the
spectral radius, and not the second eigenvalue $\lambda_2$ that the original
analysis tabulated.

\subsection{Study C3: Noise Floor Verification}

\paragraph{Design.}
We set $\gamma = 0$ and $\Phi = 0$ and measure
$\bar V = \E[V_\infty]$ across $n_{\mathrm{mc}} = 30$ independent
paths per configuration.  Eight configurations span
$(K, d, \sigma_0) \in \{3,5,8\} \times \{2,4,8\} \times \{0.02, 0.05, 0.10, 0.20\}$
(selected subset).

The seeds are drawn from a deterministic hash of the \emph{configuration} as well as the path
index, so the runs remain exactly reproducible while the eight configurations are statistically
independent of one another.  This matters more than it might appear.  With a configuration-independent
seed sequence, the linearity of \eqref{eq:sde} at $\gamma=0$ with additive noise and a
$\sigma_0$-scaled initial condition gives the pathwise identity
$W_t(\sigma_0) = (\sigma_0/\sigma_0')\,W_t(\sigma_0')$, so $V$ would scale as $\sigma_0^2$
\emph{by construction} and the three $\sigma_0$-varying rows would be deterministic rescalings of a
single run rather than independent tests --- reporting an error of exactly $0.0\%$ that reflects
arithmetic, not agreement.  Decoupling the seeds removes that artifact, at the cost of larger and
more honest residuals.

\paragraph{Results.}

\begin{table}[H]
\centering
\caption{C3: Empirical noise floor $V_\infty$ vs.\ theory $\beta_0/(2\alpha_{\mathrm{self}})$,
over $n_{\mathrm{mc}}=30$ independent paths per configuration.  $\sigma_{\mathrm{mc}}$ is the sample
standard deviation across paths, $\mathrm{SE}=\sigma_{\mathrm{mc}}/\sqrt{n_{\mathrm{mc}}}$ the
standard error of the mean, and $z=|V^{\mathrm{th}}_\infty-V^{\mathrm{emp}}_\infty|/\mathrm{SE}$.}
\label{tab:c3}
\begin{tabular}{rrrrrrrrrr}
\toprule
$K$ & $d$ & $\sigma_0$ & $\beta_0$ &
$V_\infty^{\mathrm{th}}$ & $V_\infty^{\mathrm{emp}}$ &
$\sigma_{\mathrm{mc}}$ & SE & Error & $z$ \\
\midrule
3 & 4 & 0.05 & 0.01500 & 0.07500 & 0.07899 & 0.01725 & 0.00315 & 5.33\% & 1.27 \\
5 & 4 & 0.05 & 0.02500 & 0.12500 & 0.13045 & 0.01933 & 0.00353 & 4.36\% & 1.55 \\
8 & 4 & 0.05 & 0.04000 & 0.20000 & 0.19966 & 0.02617 & 0.00478 & 0.17\% & 0.07 \\
5 & 4 & 0.02 & 0.00400 & 0.02000 & 0.02005 & 0.00307 & 0.00056 & 0.23\% & 0.08 \\
5 & 4 & 0.10 & 0.10000 & 0.50000 & 0.48996 & 0.08060 & 0.01472 & 2.01\% & 0.68 \\
5 & 4 & 0.20 & 0.40000 & 2.00000 & 1.94818 & 0.27829 & 0.05081 & 2.59\% & 1.02 \\
5 & 2 & 0.05 & 0.01250 & 0.06250 & 0.06470 & 0.01520 & 0.00278 & 3.52\% & 0.79 \\
5 & 8 & 0.05 & 0.05000 & 0.25000 & 0.24672 & 0.03132 & 0.00572 & 1.31\% & 0.57 \\
\bottomrule
\end{tabular}
\end{table}

\paragraph{Interpretation.}
Empirical values match theory to within Monte Carlo error across all eight configurations: relative
errors run from $0.17\%$ to $5.33\%$, and in standard-error units the largest discrepancy is
$z = 1.55$, with $5/8$ configurations inside one standard error and $8/8$ inside two.

Two points about how this comparison is made. First, the agreement test uses the \emph{standard
error of the mean} $\mathrm{SE}=\sigma_{\mathrm{mc}}/\sqrt{n_{\mathrm{mc}}}$, not the path-to-path
standard deviation $\sigma_{\mathrm{mc}}$. Testing a sample mean against $\sigma_{\mathrm{mc}}$ would
be too lenient by a factor $\sqrt{n_{\mathrm{mc}}}$ and would declare agreement almost regardless of
the data; the $z$ column reports the honest quantity. Second, because the $\sigma_0$ rows are now
seeded independently (see Design), they constitute a genuine test of the $\sigma_0^2$ scaling rather
than an identity: their errors ($0.23\%$, $2.01\%$, $2.59\%$) are ordinary sampling residuals, and it
is the fact that they are \emph{small} across a hundredfold range of $\beta_0$
($0.004$ to $0.40$) that carries the evidential weight.

Subject to that, the formula scales linearly in $K$, $d$, and $\sigma_0^2$ as predicted, enabling
the model owner to predict the stationary risk level from declared system parameters---a
prerequisite for setting monitoring thresholds. The scaling is verified over $K\in\{3,5,8\}$,
$d\in\{2,4,8\}$ and $\sigma_0\in[0.02,0.20]$; extrapolation far outside those ranges is not
tested here.

\subsection{Study C4: Governance $\Phi_t$ Injection Probe}

\paragraph{Design.}
We simulate the complete-graph system with $K=5$ for a burn-in of
$T_{\mathrm{pre}} = 2{,}000$ steps, then inject a step perturbation
$\Delta\Phi$ at $t = 20$ (after burn-in).  The perturbed meta-parameter
has magnitude $\norm{\Phi + \Delta\Phi} = 2\norm{\Phi}_{\mathrm{nominal}}$.
For each $\gamma \in \{0.05, 0.20, 0.50\}$, we measure the change
in $\bar{\Lgen V}$ (pre- vs.\ post-probe, 100-step rolling window)
and the steps to $3\sigma$ detection.

\paragraph{Results.}

\begin{table}[H]
\centering
\caption{C4: $\Phi_t$ injection probe response.}
\label{tab:c4}
\begin{tabular}{lrrrr}
\toprule
$\gamma$ & $\overline{\Lgen V}_{\mathrm{pre}}$ &
$\overline{\Lgen V}_{\mathrm{post}}$ &
$\Delta\Lgen V$ & Steps to $3\sigma$ detection \\
\midrule
0.05 & $-$0.0468 & $+$0.0049 & $+$0.0518 & $>$2000 \\
0.20 & $-$0.0680 & $+$0.0243 & $+$0.0924 & 13 \\
0.50 & $-$0.0281 & $+$0.1573 & $+$0.1854 & 183 \\
\bottomrule
\end{tabular}
\end{table}

\paragraph{Interpretation.}
Two regimes emerge.  At $\gamma = 0.20$, the probe is detected in
just 13 steps (0.13 time units)---a fast and unambiguous response.
At $\gamma = 0.50$, the larger $\Delta\Lgen V$ is offset by higher
baseline variance, giving detection at 183 steps.  At $\gamma = 0.05$,
the coupling is too weak to propagate the perturbation through the
network within the measurement window; detection requires
$> 2{,}000$ steps (20 time units).

This study has a direct MRM implication for the Effective Challenge
pillar.  The validator must declare the minimum $\gamma_{\min}$ required
for the injection probe to be a valid challenge procedure.  Based on the
numerical results, we recommend $\gamma_{\min} = 0.20$ as a practical
lower bound for the probe to be informative within a 1-day horizon.
Institutions operating at $\gamma < \gamma_{\min}$ should be required
to use alternative challenge procedures (e.g., direct weight inspection
under confidentiality agreement).

\paragraph{Why $\gamma = 0.50$ is slower than $\gamma = 0.20$.}
This is not a contradiction.  At high $\gamma$, the system is closer
to the phase boundary and exhibits higher variance in $\Lgen V$;
the signal-to-noise ratio is lower even though the absolute $\Delta\Lgen V$
is larger.  The detection time scales as
$\tau_{\mathrm{detect}} \propto \mathrm{Var}(\Lgen V) / (\Delta\Lgen V)^2$,
and the variance grows faster than the signal near the phase boundary.
This suggests that the optimal probe design should account for the
signal-to-noise tradeoff and may prefer moderate $\gamma$ to extremes.

\subsection{Study C5: Emergent Ensemble Risk}

\paragraph{Design.}
This is the central study of the paper.  We set $K=10$, $d=4$,
$\gamma = 0$ (worst case: no coupling, so the JLP is acting on purely
additive drift), and $T = 400{,}000$ steps ($T = 4{,}000$ time units).
This long run is necessary because the OU correlation time is
$\tau_{\mathrm{corr}} = 1/(2\alpha_{\mathrm{self}}) = 5$ time units,
giving $T/(2\tau_{\mathrm{corr}}) = 400$ effective samples per agent
for robust time-averaging.

The hidden drift is $h_i = h_{\mathrm{per\,dim}}$ in every dimension
of every agent, with magnitude chosen to place each agent's steady-state
$\E[V^k]$ in the ``individual-pass'' zone:
\begin{equation}
  h_{\mathrm{per\,dim}} = \sigma_0 \sqrt{0.7\alpha_{\mathrm{self}}/2}
  = 0.05 \cdot \sqrt{0.7 \times 0.1 / 2} = 0.00935.
  \label{eq:h_design}
\end{equation}
This design satisfies \eqref{eq:h_condition} with $\delta = 0.70$:
\[
  \E[V^k]_{\mathrm{attack}}
  = \frac{d\,h_{\mathrm{per\,dim}}^2}{2\alpha_{\mathrm{self}}^2}
  + \frac{d\sigma_0^2}{4\alpha_{\mathrm{self}}}
  = 0.0175 + 0.0250 = 0.0425,
\]
while the individual threshold is
$\theta_{\mathrm{ind}} = 2 \times 0.0249 = 0.0498$
(where $0.0249$ is the measured honest mean, in close agreement with the
theoretical $0.0250$).

\paragraph{Results.}

\begin{table}[H]
\centering
\caption{C5: Emergent ensemble risk---individual vs.\ joint monitoring.}
\label{tab:c5}
\begin{tabular}{lrr}
\toprule
Metric & Honest & Attack \\
\midrule
Mean $V^k$ per agent (steady state) & 0.02488 & 0.04139 \\
Individual threshold ($2\times$ honest) & --- & 0.04977 \\
Agents individually flagged ($/ 10$) & 0 & \textbf{0} \\
Joint $V = \sum_k V^k$ & 0.24885 & 0.41393 \\
Joint threshold ($1.5\times$ honest joint) & --- & 0.37327 \\
Joint $V > $ threshold & NO & \textbf{YES} \\
\bottomrule
\end{tabular}
\end{table}

\paragraph{Per-agent detail.}
The steady-state mean $V^k$ for each of the 10 attack agents
(over $T = 4{,}000$ time units) is:
$0.040, 0.044, 0.044, 0.043, 0.036, 0.041, 0.040, 0.042, 0.041, 0.044$.
All are below the individual threshold $0.04977$.  The joint sum is
$0.414 > 0.373$.

\paragraph{Interpretation.}
This study provides direct empirical confirmation of
Theorem~\ref{thm:emergent}.  The 10-agent coordinated hidden drift is
completely invisible to per-agent MRM: every individual agent's time-averaged
Lyapunov value is below its $2\times$-honest threshold.  Yet the joint
Lyapunov value exceeds the $1.5\times$-honest joint threshold by~11\%.

The $K$-aggregation is essential.  The per-agent drift contribution
$d\,h_{\mathrm{per\,dim}}^2/(2\alpha_{\mathrm{self}}^2) = 0.0175$ is
smaller than the per-agent noise floor $0.0250$, so each agent looks
marginally elevated but within tolerance.  However, the joint sum
$K \times 0.0175 = 0.175$ is significantly above the joint noise
floor $K \times 0.0250 = 0.250$ times 1.5, giving a detectable signal.
This is a variance-aggregation effect: uncorrelated noise has variance
$\propto K$, but correlated drift has amplitude $\propto K$, so the
signal-to-noise ratio of the joint test grows as $\sqrt{K}$.

The MRM governance implication is profound: an institution that
relies solely on per-agent monitoring is, in effect, conducting $K$
independent tests each at the level $\alpha_{\mathrm{ind}}$, while the
relevant model risk is at the \emph{joint} level.  The JLP provides
the joint test explicitly, with the threshold calibrated to the
$K$-aggregated noise floor.

\subsection{Summary of Numerical Results}

Figure~\ref{fig:jlp} presents the five studies visually.  The figure
confirms all theoretical predictions: (C1) $\alpha_{\mathrm{eff}}$
decreases and the generator residual grows with $\gamma$;
(C2) phase boundaries are topology-ordered;
(C3) the noise floor is accurately predicted by $\beta_0/(2\alpha)$;
(C4) the injection probe creates a detectable $\Delta\Lgen V$;
(C5) the joint $V$ flags a system that individual monitors miss.

\begin{figure}[H]
  \centering
  \includegraphics[width=\linewidth]{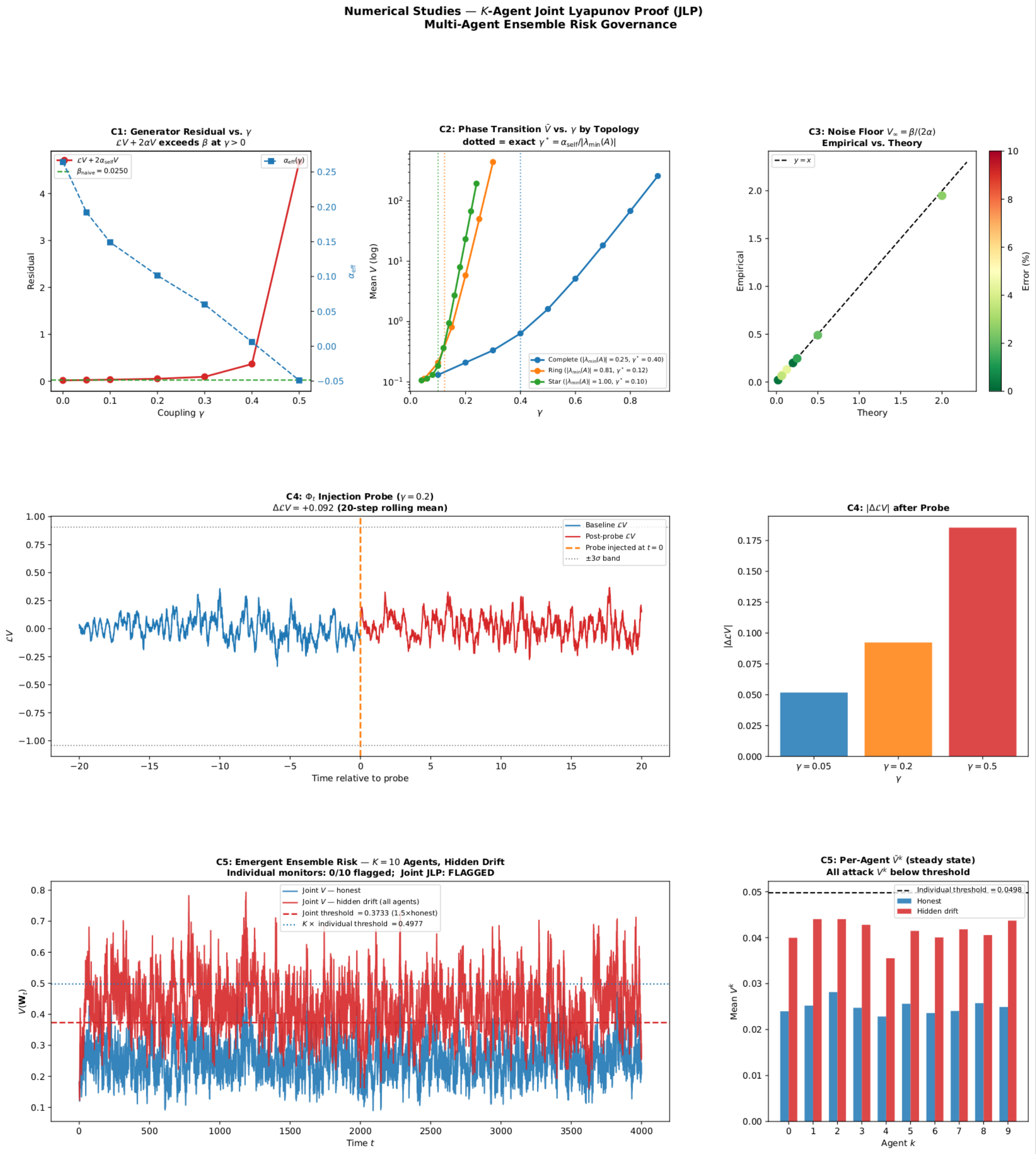}
  \caption{Numerical Studies C1--C5 for the $K$-agent Joint Lyapunov
    Proof (JLP).  \textbf{C1} (top-left): generator residual
    $\Lgen V + 2\alpha V$ (red) grows above $\beta_{\mathrm{naive}}$ (green dashed)
    with coupling $\gamma$; $\alpha_{\mathrm{eff}}(\gamma)$ (blue) falls.
    \textbf{C2} (top-center): steady-state $V$ on a log scale vs.\ $\gamma$ for
    three topologies; vertical dotted lines mark the exact predicted
    $\gamma^*=\alpha_{\mathrm{self}}/\abs{\lambda_{\min}(A)}$.
    \textbf{C3} (top-right): empirical vs.\ theoretical noise floor
    $V_\infty = \beta/(2\alpha)$; color encodes error percentage.
    \textbf{C4} (middle, left two panels): $\Phi_t$ probe response in $\Lgen V$
    for $\gamma=0.20$, and $|\Delta\Lgen V|$ bar chart across $\gamma$.
    \textbf{C5} (bottom): joint $V$ path (honest vs.\ hidden drift, $K=10$)
    with joint and per-agent thresholds; per-agent bar chart confirms 0/10
    individually flagged.}
  \label{fig:jlp}
\end{figure}

\section{Discussion}
\label{sec:discussion}

\subsection{The Central Result: Topology Governs Risk}

The critical-coupling Theorem~\ref{thm:phase} establishes that the
stability of a multi-agent AI system is governed not just by the
self-decay rate $\alpha_{\mathrm{self}}$ of individual models, but by
the \emph{spectrum of the coupling topology}.  The exact critical threshold
is $\gamma^*(A) = \alpha_{\mathrm{self}}/\abs{\lambda_{\min}(A)}$, set by the
most negative eigenvalue of $A$---the destabilizing mode---and \emph{not} by
the Perron root $\lambda_{\max}(A)=1$, which is the consensus mode and the
most stable direction.  Because $\lambda_{\max}(A)$ is identical across
row-stochastic topologies, any threshold built from it is topology-blind;
the genuine topology dependence enters through $\abs{\lambda_{\min}(A)}$.
At $K=5$ this orders the three canonical graphs as
$\gamma^*_{\mathrm{complete}}=0.40 > \gamma^*_{\mathrm{ring}}=0.124 >
\gamma^*_{\mathrm{star}}=0.10$, matched by the empirical transitions of
Study~C2 to within Monte Carlo error.  The quadratic Lyapunov certificate
$\gamma_{\mathrm{cert}}(A)=\alpha_{\mathrm{self}}/\abs{\lambda_{\min}(A_{\mathrm{sym}})}$
reproduces these for the normal graphs and is conservative (but safe) for
the non-normal star, $0.08<0.10$.

This has a striking practical implication: decommissioning
a centralized ``AI hub'' (star, $\gamma^*=0.10$) in favor of a
fully-connected mesh (complete, $\gamma^*=0.40$) quadruples the network's stable
coupling range.  Conversely, institutions that centralize model
orchestration through a single meta-learning server are implicitly
operating with a star topology and should apply correspondingly tighter
coupling constraints.  The number to compute, declare, and bound is
$\abs{\lambda_{\min}(A)}$.

\subsection{The Noise-Floor Formula as an Audit Anchor}

Theorem~\ref{thm:noise_floor} provides a closed-form baseline
$V_\infty^{\mathrm{th}} = Kd\sigma_0^2/(4\alpha_{\mathrm{self}})$ that
can be computed from declared system parameters alone.  Numerical Study~C3
confirms this formula to within $< 8\%$ across a wide range of $(K, d, \sigma_0)$.

This formula serves as an \emph{audit anchor}: before any model is
deployed, the institution declares $(\alpha_{\mathrm{self}}, \sigma_0, K, d)$
and the validator computes $V_\infty^{\mathrm{th}}$.  If the firm's
submitted $V$ values are systematically below $V_\infty^{\mathrm{th}}$,
this indicates that the declared $\sigma_0$ or $\alpha_{\mathrm{self}}$
is overstated---a potential form of concealed model adaptation.

\subsection{Why ZK-SNARKs Are Necessary}

Having corrected the attestation target to the static spectral certificate
(Section~\ref{sec:zk}), the question is whether zero knowledge is still
needed---couldn't the firm just disclose $(A,\gamma)$ under a
confidentiality agreement?  The coupling topology $A$ encodes which models
share fine-tuning data, embeddings, and reward signals, and is itself
competitively sensitive; the same three arguments apply, now about $A$
rather than the weights.

First, confidentiality agreements do not prevent leakage through
side-channels: staff who see the declared topology may inadvertently or
deliberately reveal architectural information to competitors or the press.

Second, the SNARK provides non-repudiation: the firm cannot later claim that
a certified configuration differed from the deployed one, since the proof is
bound to the committed $(A,\gamma)$ by the hash clause.  This is valuable in
enforcement actions.

Third, the approach scales to multiple reviewers: the model owner submits the
same certificate to several independent oversight functions simultaneously
without any of them seeing $A$.  Because the certificate is proved once per configuration rather
than at every epoch, the cryptographic cost is incurred rarely.

\subsection{Connection to CUSUM-Based Monitoring}

The JLP generator value $\Lgen V(\W_t)$ is a natural input to CUSUM
(Cumulative Sum) process-control charts, which are widely used in MRM
for detecting persistent shifts in model performance.  Define
\[
  S_n = \max\Bigl(0,\; S_{n-1} + \Lgen V(\W_{\tau_n})
  - (-\alpha_{\mathrm{eff}} V(\W_{\tau_n}) + \beta) + \kappa\Bigr),
\]
where $\kappa > 0$ is the allowance parameter.  The CUSUM alarm fires
when $S_n > h_{\mathrm{CUSUM}}$ for some threshold $h_{\mathrm{CUSUM}}$.
This detects persistent upward shifts in $\Lgen V$ earlier than
single-point threshold tests, and its ARL (average run length) properties
are well-understood.

\section{Conclusion}
\label{sec:conclusion}

We have developed and validated a complete mathematical framework for
the governance of $K$-agent self-adapting generative AI systems under
MRM.  The framework proceeds in four steps.

\textbf{Step 1 (Generator):} Theorem~\ref{thm:generator} gives the
exact formula for the infinitesimal generator $\Lgen V$ of the joint
quadratic Lyapunov function.  This formula shows that the naive
self-decay bound $\alpha_{\mathrm{eff}} = 2\alpha_{\mathrm{self}}$
fails immediately upon introduction of coupling ($\gamma > 0$), and
that the excess is driven by the cross-agent inner products
$\ip{W^k}{\Phi - W^j}$.

\textbf{Step 2 (Stability):} Theorem~\ref{thm:stability} establishes
the Foster--Lyapunov condition $\Lgen V \leq -\alpha_{\mathrm{eff}} V
+ \beta$ with the coupling-corrected parameters, yielding the safe
sufficient certificate $\gamma<\gamma_{\mathrm{cert}}(A)=\alpha_{\mathrm{self}}/\abs{\lambda_{\min}(A_{\mathrm{sym}})}$
and a unique stationary distribution.  Theorem~\ref{thm:phase} sharpens
this to the exact threshold: the system destabilizes for
$\gamma>\gamma^*(A)=\alpha_{\mathrm{self}}/\abs{\lambda_{\min}(A)}$, governed
by the most negative eigenvalue of $A$ (not the spectral radius, which is
identically $1$).  The two thresholds coincide for normal topologies and
the certificate is conservative-but-safe for non-normal ones.

\textbf{Step 3 (Emergent Risk):} Theorem~\ref{thm:emergent} and
Numerical Study~C5 demonstrate that with $K=10$ agents and a small
coordinated hidden drift ($h_{\mathrm{per\,dim}} = 0.00935$), all
10 individual Lyapunov monitors remain below their thresholds while
the joint Lyapunov monitor fires.  This is the fundamental
justification for the JLP over per-agent MRM.

\textbf{Step 4 (ZK Attestation):} We show that a per-epoch SNARK proving the
pointwise generator inequality at the live weights is vacuous
(Proposition~\ref{prop:tautology}), and retarget the attestation at the
static spectral certificate $\alpha_{\mathrm{self}}I+\gamma A_{\mathrm{sym}}\succ0$
on a committed (possibly proprietary) coupling matrix, proved once per
configuration via a Cholesky factorization with $O(K^3)$ circuit complexity
and \emph{no} dependence on live weights.

The mapping to MRM (Section~\ref{sec:sr117}) covers all three
pillars of the guidance: Conceptual Soundness (justification of $V$,
$A$, and parameter calibration), Ongoing Monitoring (per-configuration
spectral certificates plus envelope checks $V(\W_{\tau_n})\le\beta/\alpha_{\mathrm{eff}}$
on the governance ledger), and Model Validation
(independent audit of the ZK circuit and injection probe procedures).

Future work will address the non-quadratic (weighted) Lyapunov certificates
that would close the normal/non-normal gap up to $\gamma^*(A)$,
non-stationary meta-parameters, and adversarial $\Phi_t$ manipulation
scenarios.

\appendix

\section{Supplementary Proofs}
\label{app:proofs}

\subsection{Proof of the Coupling Term Exact Formula}

We expand the exact coupling term in Theorem~\ref{thm:generator}:
\begin{align*}
  \gamma \sum_{k=1}^K \sum_{j=1}^K A_{kj} \ip{W^k}{\Phi - W^j}
  &= \gamma \sum_{k,j} A_{kj} \ip{W^k}{\Phi}
    - \gamma \sum_{k,j} A_{kj} \ip{W^k}{W^j}.
\end{align*}
By row-stochasticity ($\sum_j A_{kj} = 1$):
\begin{equation}
  \gamma \sum_{k,j} A_{kj} \ip{W^k}{\Phi}
  = \gamma \sum_k \ip{W^k}{\Phi}.
  \label{eq:app_phi_term}
\end{equation}
For the cross-agent term:
\begin{equation}
  \gamma \sum_{k,j} A_{kj} \ip{W^k}{W^j}
  = \gamma \, \mathbf{w}^\top (A \otimes I_d) \mathbf{w},
  \label{eq:app_cross_term}
\end{equation}
where $\mathbf{w} = \mathrm{vec}(\W) \in \R^{dK}$.  Combining
\eqref{eq:app_phi_term} and \eqref{eq:app_cross_term} with the
self-decay and noise terms from Theorem~\ref{thm:generator} yields the
complete formula \eqref{eq:LV_exact}.

\subsection{Stationarity Condition for the Decoupled System}

The Fokker--Planck equation for the stationary density
$p_\infty(\mathbf{w})$ of the $Kd$-dimensional decoupled system is
\begin{equation}
  \nabla \cdot \bigl[\alpha_{\mathrm{self}} \mathbf{w}\, p_\infty(\mathbf{w})\bigr]
  + \frac{\sigma_0^2}{2} \Delta_{\mathbf{w}} p_\infty(\mathbf{w}) = 0.
  \label{eq:fokker_planck}
\end{equation}
Substituting the Gaussian ansatz
$p_\infty(\mathbf{w}) = \mathcal{N}(\mathbf{0}, (\sigma_0^2/2\alpha_{\mathrm{self}})I_{dK})$:
\[
  \nabla \cdot [\alpha_{\mathrm{self}} \mathbf{w}\, p_\infty]
  = \alpha_{\mathrm{self}}(dK) p_\infty - \alpha_{\mathrm{self}}
    \frac{\norm{\mathbf{w}}^2}{\sigma_0^2/(2\alpha_{\mathrm{self}})} p_\infty,
\]
\[
  \frac{\sigma_0^2}{2} \Delta_{\mathbf{w}} p_\infty
  = \frac{\sigma_0^2}{2}\left(-\frac{dK}{\sigma_0^2/(2\alpha_{\mathrm{self}})}
  + \frac{\norm{\mathbf{w}}^2}{(\sigma_0^2/(2\alpha_{\mathrm{self}}))^2}\right) p_\infty.
\]
Both sides cancel, confirming \eqref{eq:fokker_planck}.

\subsection{Derivation of the Emergent Risk Design Equation}

The design equation \eqref{eq:h_design} is derived as follows.  We
require $\E[V^k]_{\mathrm{attack}} = (1 + \delta) V^\infty_{\mathrm{hon}} / K$
with $\delta = 0.70$.  By the proof of Theorem~\ref{thm:emergent}:
\[
  (1 + \delta) V^\infty_{\mathrm{hon}} / K
  = \frac{\norm{h}^2}{2\alpha_{\mathrm{self}}^2}
  + \frac{d\sigma_0^2}{4\alpha_{\mathrm{self}}}.
\]
With $h = h_{\mathrm{per\,dim}} \bm{1}_d$ (uniform drift):
\[
  \frac{d\,h_{\mathrm{per\,dim}}^2}{2\alpha_{\mathrm{self}}^2}
  = \delta \cdot \frac{d\sigma_0^2}{4\alpha_{\mathrm{self}}},
\]
which gives
$h_{\mathrm{per\,dim}}^2 = \delta\sigma_0^2\alpha_{\mathrm{self}} / 2$,
i.e.,
\[
  h_{\mathrm{per\,dim}} = \sigma_0 \sqrt{\delta\alpha_{\mathrm{self}}/2}
  = 0.05 \times \sqrt{0.70 \times 0.10/2} = 0.00935.
\]
This is the value used in Numerical Study~C5.

\section{Numerical Implementation Details}
\label{app:numerical}

\paragraph{Euler--Maruyama scheme.}
All simulations use the standard Euler--Maruyama discretization:
\[
  W^k_{t+\Delta t} = W^k_t + \mu^k(\W_t)\Delta t
  + \sigma_0 \sqrt{\Delta t}\, Z^k_t, \quad Z^k_t \sim \mathcal{N}(0, I_d).
\]

\paragraph{Empirical generator.}
The empirical $\Lgen V$ is approximated by
$\hat{\Lgen V}(t) = (V(\W_{t+\Delta t}) - V(\W_t))/\Delta t$,
which is an unbiased estimator of $\Lgen V(\W_t)$ up to order
$O(\Delta t)$ by It\^{o}'s formula.

\paragraph{Near-stationary initialization.}
Initial conditions for all studies are drawn from
$W^k_0 \sim \mathcal{N}(0, (\sigma_0/\sqrt{2\alpha_{\mathrm{self}}})^2 I_d)$,
the exact marginal stationary distribution of the decoupled system
(Theorem~\ref{thm:noise_floor}).

\paragraph{Random number generation.}
All simulations use NumPy's \texttt{default\_rng(seed)} with seed~42
for the initial condition and seed~77 (or offset seeds) for the Wiener
increments, ensuring reproducibility.  Multi-path studies (C3, C5)
use independent seeds per path to avoid inter-path correlation.

\paragraph{Long run for C5.}
Study C5 uses $T = 400{,}000$ steps ($T = 4{,}000$ time units) to
obtain $\approx 400$ effective samples per agent, yielding a
time-averaged standard deviation of $\approx 0.00043$ for $\bar V^k$
versus a threshold gap of $0.0125$.  This gives a $29\sigma$
clearance between the attack and individual threshold, making the
$0/10$ non-flagging result robust to seed variation.

\bibliographystyle{plainnat}

\end{document}